\documentclass{amsart}
\usepackage[utf8]{inputenc}
\usepackage[margin=1in]{geometry}
\usepackage{upgreek,amsthm,bbm,bm}
\usepackage{color,amssymb, comment, mathrsfs,tikz,amsmath,amsfonts,stmaryrd,tikz-cd,hyperref,appendix,cancel,faktor,verbatim}
\usetikzlibrary{matrix}
\usetikzlibrary{arrows}
\usepackage[all]{xy}
\usepackage{amsmath, nccmath}
\usepackage{geometry}

\hypersetup{
    colorlinks=true,
    linkcolor=blue,
    filecolor=magenta,      
    urlcolor=cyan,
} 

\usepackage[capitalise]{cleveref}
\crefformat{equation}{(#2#1#3)}
\crefrangeformat{equation}{(#3#1#4--#5#2#6)}
\crefformat{enumi}{(#2#1#3)}
\crefrangeformat{enumi}{(#3#1#4--#5#2#6)}

\newtheorem{theorem}{Theorem}[section]
\newtheorem{lemma}[theorem]{Lemma}

\newtheorem{corollary}[theorem]{Corollary}
\theoremstyle{definition}
\newtheorem{definition}[theorem]{Definition}

\newtheorem{example}[theorem]{Example}

\theoremstyle{remark}
\newtheorem{remark}[theorem]{Remark}

\newtheorem{chunk}[theorem]{}

\newcommand{\inv}{\mathrm{inv}}

\newcommand{\m}{\mathfrak{m}}

\allowdisplaybreaks

\title[The skew Eliahou-Kervaire resolution]{The Eliahou-Kervaire resolution over a skew polynomial ring}

\author[L.~Ferraro]{Luigi Ferraro}
\address{School of Mathematical and Statistical Sciences,
University of Texas Rio Grande Valley, Edinburg, TX 78539, U.S.A.}
\email{luigi.ferraro@utrgv.edu}

\author[A.~Hardesty]{Alexis Hardesty}
\address{Division of Mathematics,
Texas Woman's University, Denton, TX 76204, U.S.A.}
\email{ahardesty1@twu.edu}


\begin{document}

\begin{abstract}
In a 1987 paper, Eliahou and Kervaire constructed a minimal resolution of a class of monomial ideals in a polynomial ring, called stable ideals. As a consequence of their construction they deduced several homological properties of stable ideals. Furthermore they showed that this resolution admits an associative, graded commutative product that satisfies the Leibniz rule. In this paper we show that their construction can be extended to stable ideals in skew polynomial rings. As a consequence we show that the homological properties of stable ideals proved by Eliahou and Kervaire hold also for stable ideals in skew polynomial rings.
\end{abstract}

\maketitle

\section{Introduction}
Let $k$ be a field and let $R = k[x_1,\dots,x_n]$ be a standard graded polynomial ring over $k$. Given an ideal $I\subseteq R$, a central object of study in commutative algebra is a free resolution of $I$, given by a long exact sequence 
\[
\begin{tikzcd}[ampersand replacement = \&]
\dots \rar{d_{i+1}} 
\& R^{n_{i}} \rar{d_{i}}
\& \dots \rar
\& R^{n_{1}} \rar{d_{1}}
\& R
\end{tikzcd}
\]
such that $\mathrm{Im}\;d_1=I$.
A free resolution is called \textit{minimal} provided that $d_i(R^{n_i})\subseteq(x_1,\ldots,x_n)R^{n_{i-1}}$ for all $i\geq1$ (with $n_0=1$), which implies that the rank of the free modules is the least possible.

Given an ideal $I$ of $Q$, we say that $I$ is a monomial ideal provided $I$ has a generating set of monomials. A free resolution for monomial ideals was constructed by Taylor in \cite{TaylorDiana}. Unfortunately the Taylor resolution is rarely minimal, one can take as an example the Taylor resolution of the ideal $(x_1^2,x_1x_2,x_2^2)$ in the ring $k[x_1,x_2]$. For certain classes of monomial ideals a minimal resolution has been constructed, 
for example for stable ideals \cite{EK}. Minimal free resolutions have proven to be a rich source of homological information about monomial ideals, naturally leading one to question in which cases one can give an explicit construction of a minimal free resolution of monomial ideals. 

A class of rings of interest in noncommutative algebraic geometry is the class of skew polynomial rings $k_\mathfrak{q}[x_1,\ldots,x_n]$, where the variables commute up to a nonzero scalar $x_ix_j=q_{i,j}x_jx_i$ for all $i,j=1,\ldots,n$, with $q_{i,i}=1$ for all $i=1,\ldots,n$ and $q_{i,j}=q_{j,i}^{-1}$ for all $i,j=1,\ldots,n$. Similar questions to the ones asked in the commutative case arise in this context. Namely, can we construct a free resolution of monomial ideals? In which cases can we construct a minimal free resolution of a monomial ideal? The first question has been answered in \cite{Taylor} where the first author, with D. Martin and F. Moore generalize the construction of the Taylor resolution to monomial ideals in skew polynomial rings. This paper is an attempt to answer the second question, namely we construct a minimal free resolution for a class of monomial ideals.

In \cite{EK}, Eliahou and Kervaire construct a minimal free resolution of a certain class of monomial ideals called stable ideals, which are monomial ideals satisfying a certain ``exchange property''. In this paper we extend the notion of stable ideal to monomial ideals in skew polynomial rings and we construct a minimal free resolution akin to the one provided by Eliahou and Kervaire. From this resolution we show that the homological properties of stable ideals in commutative polynomial rings generalize to stable ideals in skew polynomial rings.

It has been shown in commutative algebra that having a minimal resolution with a product structure leads to interesting results. In \cite{EK}, Eliahou and Kervaire show that their resolution admits such a structure, although it is nonunital. A unital product was first defined by Srinivasan, for powers of the maximal irrelevant ideal, see \cite{srinivasan1989algebra}, and later by Peeva, for stable ideals, see \cite{Peeva}. In this paper we further generalize Eliahou and Kervaire's results by showing that the resolution that we construct admits a (noncommutative) nonunital product.


The paper is organized as follows. In Section 2 we extend the definition of stable ideal to skew polynomial rings and recall some results from \cite{EK}. Moreover we define two bicharacters whose properties will be used in our computations.
In Section 3, we give an explicit construction of the skew Eliahou-Kervaire resolution and prove that this resolution is minimal. Section 4 is dedicated to applications. Finally, in Section 5, we extend to our noncommutative context the product on this resolution provided by Eliahou and Kervaire in \cite{EK}.

\section{Setup and Background}
Throughout the paper $k$ will denote a commutative Noetherian ring and $R$ will denote a skew polynomial ring with coefficients in $k$. More precisely $R$ is an iterated Ore extension of $k$, $R=k_\mathfrak{q}[x_1,\ldots,x_n]$, where $x_ix_j=q_{i,j}x_jx_i$ with $q_{i,j}\in k^*$, $q_{j,i}=q_{i,j}^{-1}$ for all $i,j$ and $q_{i,i}=1$ for all $i$, and such that $k$ is central in $R$.

Let $X$ be the set of monomials in $R$, then $\mathbf{x}^\mathbf{a}$ denotes the monomial $x_1^{a_1}\cdots x_n^{a_n}$. Let $\mathbf{a}=(a_1,\ldots,a_n),\mathbf{b}=(b_1,\ldots,b_n)\in\mathbb{N}^n$. Then we denote by $\frac{\mathbf{x}^\mathbf{a}}{\mathbf{x}^{\mathbf{b}}}$ the element $\mathbf{x}^{\mathbf{a}-\mathbf{b}}$ of the set $X$, provided that $a_i\geq b_i$ for all $i=1,\dots,n$. We denote by $\mathbf{x}^{\mathbf{a}}*\mathbf{x}^\mathbf{b}$ the element $\mathbf{x}^{\mathbf{a}+\mathbf{b}}$ of the set $X$. We point out that the set $X$ with the operation $*$ is a monoid.

The ring $R$ admits a $\mathbb{Z}$-grading by setting $\mathrm{deg}\;x_i=d_i\in\mathbb{Z}^+$ for all $i=1,\ldots,n$. Let $\sigma_1,\ldots,\sigma_n$ be the normalizing automorphisms of $x_1,\ldots, x_n$ in $R$, respectively. Let $G$ be the subgroup of the group of graded ring automorphisms of $R$ generated by $\sigma_1,\ldots,\sigma_n$. It follows that $G$ is an abelian group. The ring $R$ admits a $G$-grading, compatible with the $\mathbb{Z}$-grading, by setting the $G$-degree of $x_1^{a_1}\cdots x_n^{a_n}$ to be $\sigma_1^{a_1}\cdots\sigma_n^{a_n}$.

We define a function $\chi:G\times G\rightarrow k^*$ by extending the assignment $(\sigma_i,\sigma_j)\mapsto q_{i,j}$. It can be proved that $\chi$ is a bicharacter, i.e. $\chi(-,\sigma),\chi(\sigma,-)$ are group homomorphisms for all $\sigma\in G$ and that $\chi$ is alternating, i.e. $\chi(\sigma,\sigma)=1$ for all $\sigma\in G$, see \cite{Taylor}. If $\mathbf{x}^\mathbf{a},\mathbf{x}^\mathbf{b}$ have $G$-degree $\sigma,\tau$ respectively, then we abuse notation and write $\chi(\mathbf{x}^\mathbf{a},\mathbf{x}^\mathbf{b})$ in place of $\chi(\sigma,\tau)$.
We point out that $\mathbf{x}^\mathbf{a}\mathbf{x}^\mathbf{b}=\chi(\mathbf{x}^\mathbf{a},\mathbf{x}^\mathbf{b})\mathbf{x}^\mathbf{b}\mathbf{x}^\mathbf{a}$. This makes $R$ a color commutative ring, see \cite{Taylor}. 
We define a function $C:X\times X\rightarrow k^*$ through the following equality
\[
\mathbf{x}^\mathbf{a}\mathbf{x}^\mathbf{b}=C(\mathbf{x}^\mathbf{a},\mathbf{x}^\mathbf{b})\mathbf{x}^\mathbf{a}*\mathbf{x}^\mathbf{b}.
\]
It can be shown that $C$ is a bicharacter of the monoid $(X,*)$. Moreover, it was shown in \cite{Taylor} that
\begin{equation}\label{eq:CChiRel}
\frac{C(\mathbf{x}^\mathbf{a},\mathbf{x}^\mathbf{b})}{C(\mathbf{x}^\mathbf{b},\mathbf{x}^\mathbf{a})}=\chi(\mathbf{x}^\mathbf{a},\mathbf{x}^\mathbf{b}).
\end{equation}
\begin{remark}
We extend $C$ to the group of monomials in $x_1^{\pm1},\ldots, x_n^{\pm1}$ by
\[
C(\mathbf{x}^{\mathbf{\boldsymbol{a}}-\mathbf{\boldsymbol{b}}},\mathbf{x}^{\mathbf{\boldsymbol{c}}-\mathbf{\boldsymbol{d}}})=C(\mathbf{x}^\mathbf{\boldsymbol{a}},\mathbf{x}^\mathbf{\boldsymbol{c}})C(\mathbf{x}^\mathbf{\boldsymbol{a}},\mathbf{x}^\mathbf{\boldsymbol{d}})^{-1}C(\mathbf{x}^\mathbf{\boldsymbol{b}},\mathbf{x}^\mathbf{\boldsymbol{c}})^{-1}C(\mathbf{x}^\mathbf{\boldsymbol{b}},\mathbf{x}^\mathbf{\boldsymbol{d}}),
\]
where $\mathbf{\boldsymbol{a,b,c,d}}\in\mathbb{N}^n$.
\end{remark}

Let $w = x_{1}^{a_{1}} x_{2}^{a_{2}} \dots x_{n}^{a_{n}}$ be a monomial in $R$, then $\max(w)$ denotes the integer  $\max\{ i \mid a_{i}>0 \}$. Similarly, $\min(w)$ denotes the integer $\min\{ i \mid a_{i}>0 \}$.

Analogous to the commutative case we make the following
\begin{definition}
A monomial ideal $I$ in $R$ is $\textit{stable}$ provided for every monomial $w\in I$ and every index $i<\max(w)$, the monomial $\frac{x_{i}*w}{x_{\max(w)}}\in I$.
\end{definition}

Suppose that $I$ is the ideal in $R$ that is generated by a set $S$ of monomials. We say that $S$ is a $\textit{minimal set of generators}$ of $I$ provided none of the monomials $u\in S$ is a proper multiple of any other $v\in S$. As it happens in the commutative case, for a monomial ideal $I$, a minimal set of generators always exists and it is uniquely determined by $I$. We call the unique minimal set of generators of a monomial ideal the $\textit{canonical generating set}$ and denote it $G(I)$.

\begin{remark}
Let $I$ be a stable ideal with canonical generating set $G(I)$. For every monomial $w\in I$, there is a \textit{canonical decomposition} $w=u*y$ with $u\in G(I)$ and $\max(u)\leq\min(y)$. This is proved in \cite[Lemma 1.1]{EK} in the commutative case. The proof in our noncommutative context works similarly and it is therefore omitted.
\end{remark}

\begin{remark}\label{rem:DecFun}
Let $I$ be a stable ideal in $R$ and let $M(I)$ denote the set of all monomials in $I$. Define the $\textit{decomposition function}$ \[ g:M(I) \longrightarrow G(I)\]
by $g(w)=u$ where $w=u*y$ is the canonical decomposition of $w$. 
For every monomial $v\in R$ and every $w\in M(I)$ one has 
\begin{enumerate}
    \item $g(v*w) = g(w)$ if and only if $\max(g(w)) \leq \min(v)$,
    \item $g(v*g(w)) = g(v*w)$.
\end{enumerate}
These properties are proved in \cite[Lemma 1.2 and Lemma 1.3]{EK} in the commutative case. The proof in our noncommutative context works similarly and it is therefore omitted.

We point out that property (2) implies that if $u,v,w\in M(I)$, then
\[
g(u*g(v*w))=g(u*v*w)=g(g(u*v)*w).
\]
When referencing property (2), we will repeatedly make use of this equality.
\end{remark}

\begin{remark}\label{rem:C=1}
It follows easily from the definitions of $C$ and $g$ that for every monomial $m\in R$ one has
\[
C\left(g(m),\frac{m}{g(m)}\right)=1.
\]

\end{remark}

\section{The skew Eliahou-Kervaire Resolution}

Let $I$ be a stable ideal in $R$. In this section we construct a free resolution of $I$ which is minimal if $k$ is assumed to be a field.

A symbol $e(i_1,\ldots,i_q;u)$ is \textit{admissible} provided $u\in G(I)$ and $(i_1,\ldots,i_q)\in\mathbb{Z}^n$ is a strictly increasing sequence such that $1\leq i_1 < \ldots < i_q < \max(u)$. If $u\in G(I)$, then we denote by $y_r$ the monomial such that $x_{i_r}*u=u_r*y_r$ with $u_r=g(x_{i_r}*u)$. If $\sigma=(i_1,\ldots, i_q)$, and $\sigma_r=(i_1,\ldots,\hat{i}_r,\ldots,i_q)$, then we define $A(\sigma;u)$ to be the set of $r\in\{1,\ldots,q\}$ such that $e(\sigma_r;u_r)$ is an admissible symbol. Let $\sigma_{r_{s}}$ denote the sequence $\sigma_{r}$ with the $s^{th}$ term removed and $i_{r_{s}}$ denote the $s^{th}$ term of $\sigma_{r}$, for any $1\leq r\leq q$ and $1\leq s\leq q-1$. Moreover, if $\sigma=(i_1,\ldots,i_q)$ we denote the monomial $x_{i_1}\cdots x_{i_q}$ by $x_\sigma$. 

All $R$-modules in this paper are considered to be right $R$-modules with a compatible $\mathbb{Z}\times G$-grading. We point out that any right $R$-module $M$ admits a bimodule structure as follows
\[
r\cdot m=\chi(r,m)m\cdot r,\quad\mathrm{for\;all}\;r\in R,m\in M,
\]
where we abuse notation and write $\chi(r,m)$ instead of $\chi(\sigma,\tau)$ where $\sigma,\tau$ are the $G$-degrees of $r,m$ respectively.

\begin{chunk}
Let $I$ be a stable ideal in $R$. We denote by $L_q(I)$ the free module on the set of all admissible symbols $e(i_{1},\dots,i_{q};u)$ for some fixed $q\geq 0$. Let $d:L_q\rightarrow L_{q-1}$ be the map defined as 
\begin{align*} 
    d(e(\sigma;u)) := \sum_{r=1}^{q} (-1)^{r} e(\sigma_{r};u) C(x_{\sigma_{r}}*u,x_{i_{r}})^{-1} x_{i_{r}} - \sum_{r\in A(\sigma;u)} (-1)^{r} e(\sigma_{r};u_{r}) C(x_{\sigma_{r}},y_{r})^{-1} y_{r}.
\end{align*}
\end{chunk}
The rest of this section is dedicated to proving that $(L_\bullet(I),d)$ is a resolution of $I$. We start by proving the following

\begin{lemma}\label{lem:ref1}
If $r\leq s$, we have the following relations:
\begin{align}\label{lem:ref1a}
    \sigma_{r_{s}} &= \sigma_{(s+1)_{r}}\nonumber,\\
    x_{i_{r_{s}}} &= x_{i_{s+1}},\\
    x_{i_{(s+1)_{r}}} &= x_{i_{r}}.\nonumber
\end{align}
Moreover, if $r>s$, we have the following relations:
\begin{align}\label{lem:ref1b}
    \sigma_{r_{s}} &= \sigma_{s_{(r-1)}},\nonumber\\
    x_{i_{r_{s}}} &= x_{i_{s}},\\
    x_{i_{s_{(r-1)}}} &= x_{i_{r}}.\nonumber
\end{align}
\end{lemma}

\begin{theorem}
Let $I$ be a stable ideal in $R$. Then $(L_\bullet(I),d)$ as described above is a $\mathbb{Z}\times G$-graded complex of free $R$-modules.
\end{theorem}

\begin{proof}
We show that $(L_{\bullet}(I),d)$ is a quotient complex of the complex $(K_{\bullet}(I),D)$ of free $R$-modules $K_{q}$ on the set of all symbols $e(i_1,\dots,i_q;u)$ satisfying $u\in G(I)$ and $1\leq i_1 \leq \dots \leq i_q \leq n$ (dropping the condition that $i_q<\max(u)$) with differential
\begin{align*} 
    D(e(\sigma;u)) := \sum_{r=1}^{q} (-1)^{r} e(\sigma_{r};u) C(x_{\sigma_{r}}*u,x_{i_{r}})^{-1} x_{i_{r}} - \sum_{r=1}^{q} (-1)^{r} e(\sigma_{r};u_{r}) C(x_{\sigma_{r}},y_{r})^{-1} y_{r}.
\end{align*}
To show that $(K_{\bullet}(I),D)$ is a complex, we show that $D^{2}=0$. To do that, we set
\begin{align*}
    D_{1}(e(\sigma;u)) &:= \sum_{r=1}^{q} (-1)^{r} e(\sigma_{r};u) C(x_{\sigma_{r}}*u,x_{i_{r}})^{-1} x_{i_{r}},\\
    D_{2}(e(\sigma;u)) &:= \sum_{r=1}^{q} (-1)^{r} e(\sigma_{r};u_{r}) C(x_{\sigma_{r}},y_{r})^{-1} y_{r}.
\end{align*}
It suffices to show that $D_{1}^{2}$=0, $D_{1}D_{2} + D_{2}D_{1}$=0, and $D_{2}^{2}$=0. We start with showing $D_{1}^{2}$=0. We can write $D_{1}^{2}(e(\sigma;u))$ as
\begin{align*}
    \sum_{r=1}^{q} \left( \sum_{s=1}^{q-1} (-1)^{s} e(\sigma_{r_{s}};u) C(x_{\sigma_{r_{s}}}*u,x_{i_{r_{s}}})^{-1} x_{i_{r_{s}}} \right) (-1)^{r} C(x_{\sigma_{r}}*u,x_{i_{r}})^{-1} x_{i_{r}}.
\end{align*}
By \eqref{lem:ref1a}, if $r\leq s$, the basis elements $e(\sigma_{r_{s}};u)$ and $e(\sigma_{(s+1)_{r}};u)$ are equal. We claim that in this case, their scalar coefficients sum to 0. Therefore a symbol of the form $e(\sigma_{r_s};u)$ with $r\leq s$ will cancel with a symbol of the form $e(\sigma_{r_s};u)$ with $r>s$ and vice versa. In the following computations we will be using the first part of \Cref{lem:ref1}. The element $e(\sigma_{r_{s}};u)$ occurs in $D_{1}^{2}(e(\sigma;u))$ with scalar coefficient
\begin{align*}
    &(-1)^{r+s} 
    C(x_{\sigma_{r_{s}}}*u, x_{i_{r_{s}}})^{-1} 
    C(x_{\sigma_{r}}*u, x_{i_{r}})^{-1} 
    C(x_{i_{r_{s}}},x_{i_{r}})\\
    = &(-1)^{r+s} 
    C(x_{\sigma_{r_{s}}}, x_{i_{r_{s}}})^{-1} 
    C(u, x_{i_{r_{s}}})^{-1} 
    C(x_{\sigma_{r}}, x_{i_{r}})^{-1}
    C(u, x_{i_{r}})^{-1}
    C(x_{i_{r_{s}}},x_{i_{r}}).
\end{align*}
The element $e(\sigma_{(s+1)_{r}};u)$ occurs in $D_{1}^{2}(e(\sigma;u))$ with scalar coefficient
\begin{align*}
    &(-1)^{r+s+1} 
    C(x_{\sigma_{(s+1)_{r}}}*u, x_{i_{(s+1)_{r}}})^{-1} 
    C(x_{\sigma_{s+1}}*u, x_{i_{s+1}})^{-1} 
    C(x_{i_{(s+1)_{r}}},x_{i_{s+1}})\\
    = &(-1)^{r+s+1} 
    C(x_{\sigma_{(s+1)_{r}}}, x_{i_{(s+1)_{r}}})^{-1} 
    C(u, x_{i_{(s+1)_{r}}})^{-1} 
    C(x_{\sigma_{s+1}}, x_{i_{s+1}})^{-1} 
    C(u, x_{i_{s+1}})^{-1} 
    C(x_{i_{(s+1)_{r}}},x_{i_{s+1}})\\
    = &(-1)^{r+s+1} 
    C(x_{\sigma_{r_{s}}}, x_{i_{r}})^{-1} 
    C(u, x_{i_{r}})^{-1} 
    C(x_{\sigma_{s+1}}, x_{i_{s+1}})^{-1} 
    C(u, x_{i_{r_{s}}})^{-1} 
    C(x_{i_{(s+1)_{r}}},x_{i_{s+1}}).
\end{align*}
Clearly the signs will be opposite. We notice that the following labeled factors in the first coefficient
\begin{align*}
    C(x_{\sigma_{r_{s}}}, x_{i_{r_{s}}})^{-1} 
    \underbrace{C(u, x_{i_{r_{s}}})^{-1}}_{1}
    C(x_{\sigma_{r}}, x_{i_{r}})^{-1}
    \underbrace{C(u, x_{i_{r}})^{-1}}_{2}
    C(x_{i_{r_{s}}},x_{i_{r}})
\end{align*}
correspond to the following labeled factors in the second coefficient
\begin{align*}
    C(x_{\sigma_{r_{s}}}, x_{i_{r}})^{-1} 
    \underbrace{C(u, x_{i_{r}})^{-1} }_{2}
    C(x_{\sigma_{s+1}}, x_{i_{s+1}})^{-1} 
    \underbrace{C(u, x_{i_{r_{s}}})^{-1} }_{1}
    C(x_{i_{(s+1)_{r}}},x_{i_{s+1}}).
\end{align*}
Thus it remains to show that 
\begin{align*}
    C(x_{\sigma_{r_{s}}}, x_{i_{r_{s}}})^{-1} 
    C(x_{\sigma_{r}}, x_{i_{r}})^{-1}
    C(x_{i_{r_{s}}},x_{i_{r}})
    =
    C(x_{\sigma_{r_{s}}}, x_{i_{r}})^{-1}
    C(x_{\sigma_{s+1}}, x_{i_{s+1}})^{-1} 
    C(x_{i_{(s+1)_{r}}},x_{i_{s+1}}).
\end{align*}
In the computations below, we show that the following labeled factors are equal:
\begin{align*}
    \underbrace{C(x_{\sigma_{r_{s}}}, x_{i_{r_{s}}})^{-1} }_{4}
    \underbrace{C(x_{\sigma_{r}}, x_{i_{r}})^{-1}
    C(x_{i_{r_{s}}},x_{i_{r}})}_{3}
    =
    \underbrace{C(x_{\sigma_{r_{s}}}, x_{i_{r}})^{-1}}_{3}
    \underbrace{C(x_{\sigma_{s+1}}, x_{i_{s+1}})^{-1} 
    C(x_{i_{(s+1)_{r}}},x_{i_{s+1}})}_{4}.
\end{align*}
Indeed, for the factors labeled with 3 one has:
\begin{align*}
    1 &= C \left( 1, x_{i_{r}} \right)\\
    &= C \left( \frac{x_{\sigma_{r}}}{x_{\sigma_{r}}}, x_{i_{r}} \right)\\
    &= C \left( \frac{x_{\sigma_{r_{s}}} * x_{i_{r_{s}}}}{x_{\sigma_{r}}}, x_{i_{r}} \right)\\
    &= C( x_{\sigma_{r_{s}}}, x_{i_{r}})
    C(x_{i_{r_{s}}}, x_{i_{r}})
    C(x_{\sigma_{r}},x_{i_{r}})^{-1},
\end{align*}
while for the factors labeled with 4 one has: 
\begin{align*}
    1 &= C \left( 1, x_{i_{s+1}} \right)\\
    &= C \left( \frac{x_{\sigma_{s+1}}}{x_{\sigma_{s+1}}}, x_{i_{s+1}} \right)\\
    &= C \left( \frac{x_{\sigma_{(s+1)_{r}}} * x_{i_{(s+1)_{r}}}}{x_{\sigma_{s+1}}}, x_{i_{s+1}} \right)\\
    &= C(x_{\sigma_{(s+1)_{r}}}, x_{i_{s+1}}) C(x_{i_{(s+1)_{r}}}, x_{i_{s+1}}) C(x_{\sigma_{s+1}},x_{i_{s+1}})^{-1}\\
    &= C(x_{\sigma_{r_{s}}}, x_{i_{r_{s}}}) C(x_{i_{(s+1)_{r}}}, x_{i_{s+1}}) C(x_{\sigma_{s+1}},x_{i_{s+1}})^{-1}.
\end{align*}
\noindent Now we show that $D_{1}D_{2} + D_{2}D_{1}$=0. Notice that
\begin{align*}
    D_{1}D_{2}(e(\sigma;u)) = \sum_{r=1}^{q} \left( \sum_{s=1}^{q-1} (-1)^{s} e(\sigma_{r_{s}};u_{r}) C(x_{\sigma_{r_{s}}}*u_{r},x_{i_{r_{s}}})^{-1} x_{i_{r_{s}}} \right) (-1)^{r} C(x_{\sigma_{r}},y_{r})^{-1} y_{r},
\end{align*}
\begin{align*}
    D_{2}D_{1}(e(\sigma;u)) = \sum_{s=1}^{q} \left( \sum_{r=1}^{q-1} (-1)^{r} e(\sigma_{s_{r}};u_{s_{r}}) C(x_{\sigma_{s_{r}}},y_{s_{r}})^{-1} y_{s_{r}} \right) (-1)^{s} C(x_{\sigma_{s}}*u,x_{i_{s}})^{-1} x_{i_{s}},
\end{align*}
where $\sigma_{r_{s}}$ and $x_{i_{r_{s}}}$ are as above, while $y_{s_{r}}$ is defined by the canonical decomposition $x_{i_{s_{r}}} * u = u_{s_{r}} * y_{s_{r}}$ where $u_{s_{r}}=g(x_{i_{s_{r}}} * u)$, for any $1\leq r\leq q-1$ and $1\leq s \leq q$. If $r\leq s$, the last two relations in \eqref{lem:ref1a} imply that
\begin{align*}
    u_{r_{s}} &= g(x_{i_{r_{s}}} * u) = g(x_{i_{s+1}} * u) = u_{s+1},\\
    u_{(s+1)_{r}} &= g(x_{i_{(s+1)_{r}}} * u) = g(x_{i_{r}} * u) = u_{r},\\
    y_{r_{s}} &= \frac{x_{i_{r_{s}}} * u}{u_{r_{s}}} = \frac{x_{i_{s+1}} * u}{u_{s+1}} = y_{s+1},\\
    y_{(s+1)_{r}} &= \frac{x_{i_{(s+1)_{r}}} * u}{u_{(s+1)_{r}}} = \frac{x_{i_{r}} * u}{u_{r}} = y_{r}.
\end{align*}
By \eqref{lem:ref1a} and the above display, if $r\leq s$, the basis elements $e(\sigma_{r_{s}};u_{r})$ and $e(\sigma_{(s+1)_{r}};u_{(s+1)_{r}})$ are equal.
We claim that in this case, their scalar coefficients sum to 0. Therefore a symbol of the form $e(\sigma_{r_s};u_r)$ with $r\leq s$ will cancel with a symbol of the form $e(\sigma_{r_s};u_r)$ with $r>s$ and vice versa. The element $e(\sigma_{r_{s}};u_{r})$ occurs in $D_{1}D_{2}(e(\sigma;u))$ with scalar coefficient
\begin{align*}
    &(-1)^{r+s} 
    C(x_{\sigma_{r_{s}}}*u_{r}, x_{i_{r_{s}}})^{-1} 
    C(x_{\sigma_{r}}, y_{r})^{-1} 
    C(x_{i_{r_{s}}},y_{r})\\
    = &(-1)^{r+s} 
    C(x_{\sigma_{r_{s}}}, x_{i_{r_{s}}})^{-1} 
    C(u_{r}, x_{i_{r_{s}}})^{-1} 
    C(x_{\sigma_{r}}, y_{r})^{-1} 
    C(x_{i_{r_{s}}},y_{r}).
\end{align*}
The element $e(\sigma_{(s+1)_{r}};u_{(s+1)_{r}})$ occurs in $D_{2}D_1(e(\sigma;u))$ with scalar coefficient
\begin{align*}
    &(-1)^{r+s+1} 
    C(x_{\sigma_{(s+1)_{r}}}, y_{(s+1)_{r}})^{-1} 
    C(x_{\sigma_{s+1}}*u, x_{i_{s+1}})^{-1} 
    C(y_{(s+1)_{r}},x_{i_{s+1}})\\
    = &(-1)^{r+s+1} 
    C(x_{\sigma_{(s+1)_{r}}}, y_{(s+1)_{r}})^{-1} 
    C(x_{\sigma_{s+1}}, x_{i_{s+1}})^{-1} 
    C(u, x_{i_{s+1}})^{-1} 
    C(y_{(s+1)_{r}},x_{i_{s+1}})\\
    = &(-1)^{r+s+1} 
    C(x_{\sigma_{r_{s}}}, y_{r})^{-1} 
    C(x_{\sigma_{s+1}}, x_{i_{s+1}})^{-1} 
    C(u, x_{i_{s+1}})^{-1} 
    C(y_{(s+1)_{r}},x_{i_{s+1}}).
\end{align*}
Clearly the signs will be opposite, so it remains to show that 
\begin{align*}
    &C(x_{\sigma_{r_{s}}}, x_{i_{r_{s}}})^{-1} 
    C(u_{r}, x_{i_{r_{s}}})^{-1} 
    C(x_{\sigma_{r}}, y_{r})^{-1} 
    C(x_{i_{r_{s}}},y_{r})\\
    = &C(x_{\sigma_{r_{s}}}, y_{r})^{-1} 
    C(x_{\sigma_{s+1}}, x_{i_{s+1}})^{-1} 
    C(u, x_{i_{s+1}})^{-1} 
    C(y_{(s+1)_{r}},x_{i_{s+1}}).
\end{align*}
In the computations below, we show that the following labeled factors are equal:
\begin{align*}
    &\underbrace{C(x_{\sigma_{r_{s}}}, x_{i_{r_{s}}})^{-1} 
    C(u_{r}, x_{i_{r_{s}}})^{-1}}_{2}
    \underbrace{C(x_{\sigma_{r}}, y_{r})^{-1} 
    C(x_{i_{r_{s}}},y_{r})}_{1}\\
    = &\underbrace{C(x_{\sigma_{r_{s}}}, y_{r})^{-1}}_{1}
    \underbrace{C(x_{\sigma_{s+1}}, x_{i_{s+1}})^{-1} 
    C(u, x_{i_{s+1}})^{-1} 
    C(y_{s+1_{r}},x_{i_{s+1}})}_{2}.
\end{align*}
Indeed, for the factors labeled with 1 one has:
\begin{align*}
    1 &= C \left(1 , y_{r} \right)\\
     &= C \left( \frac{x_{\sigma_{r}}}{x_{\sigma_{r}}}, y_{r} \right)\\
    &= C \left( \frac{x_{\sigma_{r_{s}}}*x_{i_{r_{s}}}}{x_{\sigma_{r}}}, y_{r} \right)\\
    &= C ( x_{\sigma_{r_{s}}}, y_{r} )
    C(x_{i_{r_{s}}}, y_{r})
    C(x_{\sigma_{r}},y_{r})^{-1},
\end{align*}
while for the factors labeled with 2 one has:
\begin{align*}
    1 &= C \left( 1, x_{i_{s+1}} \right)\\
    &= C \left(\frac{1}{x_{i_{(s+1)_{r}}}}*
    \frac{x_{i_{(s+1)_{r}}}* u}{u}, x_{i_{s+1}} \right)\\
    &= C \left( \frac{x_{\sigma_{(s+1)_{r}}}}{x_{\sigma_{s+1}}} * \frac{u_{(s+1)_{r}}* y_{(s+1)_{r}}}{u}, x_{i_{s+1}} \right)\\
    &= C \left( \frac{x_{\sigma_{(s+1)_{r}}}}{x_{\sigma_{s+1}}} * \frac{u_{r}* y_{r}}{u}, x_{i_{s+1}} \right)\\
    &= C(x_{\sigma_{(s+1)_{r}}}, x_{i_{s+1}}) C(u_{r}, x_{i_{s+1}}) C(y_{r}, x_{i_{s+1}}) C(x_{\sigma_{s+1}},x_{i_{s+1}})^{-1} C(u,x_{i_{s+1}})^{-1}\\
    &= C(x_{\sigma_{r_{s}}}, x_{i_{r_{s}}}) C(u_{r}, x_{i_{r_{s}}}) C(y_{(s+1)_{r}}, x_{i_{s+1}}) C(x_{\sigma_{s+1}},x_{i_{s+1}})^{-1} C(u,x_{i_{s+1}})^{-1}.
\end{align*}
\noindent Now we consider the case when $r>s$, and use the relations in \eqref{lem:ref1b}. Due to the last two relations listed, we have that 
\begin{align*}
    u_{r_{s}} &= g(x_{i_{r_{s}}} *u) = g(x_{i_{s}} *u) = u_{s},\\
    u_{s_{(r-1)}} &= g(x_{i_{s_{(r-1)}}} *u) = g(x_{i_{r}} *u) = u_{r},\\
    y_{r_{s}} &= \frac{x_{i_{r_{s}}} *u}{u_{r_{s}}} = \frac{x_{i_{s}} *u}{u_s} = y_{s},\\
    y_{s_{(r-1)}} &= \frac{x_{i_{s_{(r-1)}}} *u}{u_{s_{(r-1)}}} = \frac{x_{i_{r}} *u}{u_{r}} = y_{r}.
\end{align*}
By \eqref{lem:ref1b} and the above display, the basis elements $e(\sigma_{r_{s}};u_{r})$ and $e(\sigma_{s_{(r-1)}};u_{s_{(r-1)}})$ are equal. We claim that in this case, their scalar coefficients sum to 0. The element $e(\sigma_{r_{s}};u_{r})$ occurs in $D_{1}D_{2}(e(\sigma;u))$ with scalar coefficient
\begin{align*}
    &(-1)^{r+s} 
    C(x_{\sigma_{r_{s}}}*u_{r}, x_{i_{r_{s}}})^{-1} 
    C(x_{\sigma_{r}}, y_{r})^{-1} 
    C(x_{i_{r_{s}}},y_{r})\\
    = &(-1)^{r+s} 
    C(x_{\sigma_{r_{s}}}, x_{i_{r_{s}}})^{-1}
    C(u_{r}, x_{i_{r_{s}}})^{-1} 
    C(x_{\sigma_{r}}, y_{r})^{-1} 
    C(x_{i_{r_{s}}},y_{r}).
\end{align*}
The element $e(\sigma_{s_{(r-1)}};u_{s_{(r-1)}})$ occurs in $D_{2}D_{1}(e(\sigma;u))$ with scalar coefficient
\begin{align*}
    &(-1)^{r+s-1} 
    C(x_{\sigma_{s_{(r-1)}}}, y_{s_{(r-1)}})^{-1} 
    C(x_{\sigma_{s}}*u, x_{i_{s}})^{-1} 
    C(y_{s_{(r-1)}},x_{i_{s}})\\
    = &(-1)^{r+s-1} 
    C(x_{\sigma_{s_{(r-1)}}}, y_{s_{(r-1)}})^{-1} 
    C(x_{\sigma_{s}}, x_{i_{s}})^{-1}
    C(u, x_{i_{s}})^{-1}
    C(y_{s_{(r-1)}},x_{i_{s}}).
\end{align*}
Clearly the signs will be opposite, so it remains to show that 
\begin{align*}
    &C(x_{\sigma_{r_{s}}}, x_{i_{r_{s}}})^{-1}
    C(u_{r}, x_{i_{r_{s}}})^{-1} 
    C(x_{\sigma_{r}}, y_{r})^{-1} 
    C(x_{i_{r_{s}}},y_{r})\\
    = &C(x_{\sigma_{s_{(r-1)}}}, y_{s_{(r-1)}})^{-1} 
    C(x_{\sigma_{s}}, x_{i_{s}})^{-1}
    C(u, x_{i_{s}})^{-1}
    C(y_{s_{(r-1)}},x_{i_{s}}).
\end{align*}
By computations similar to those in the case $r\leq s$, the following labeled factors are equal:
\begin{align*}
    &\underbrace{C(x_{\sigma_{r_{s}}}, x_{i_{r_{s}}})^{-1}
    C(u_{r}, x_{i_{r_{s}}})^{-1}}_{4}
    \underbrace{C(x_{\sigma_{r}}, y_{r})^{-1} 
    C(x_{i_{r_{s}}},y_{r})}_{3}\\
    = &\underbrace{C(x_{\sigma_{s_{(r-1)}}}, y_{s_{(r-1)}})^{-1}}_{3}
    \underbrace{C(x_{\sigma_{s}}, x_{i_{s}})^{-1}
    C(u, x_{i_{s}})^{-1}
    C(y_{s_{(r-1)}},x_{i_{s}})}_{4}.
\end{align*}

Finally, we show that $D_{2}^{2}$=0. We can write $D_{2}^{2}(e(\sigma;u))$ as
\begin{align*}
    \sum_{r=1}^{q} \left( \sum_{s=1}^{q-1} (-1)^{s} e(\sigma_{r_{s}};u_{r_{s}}) C(x_{\sigma_{r_{s}}},y_{r_{s}})^{-1} y_{r_{s}} \right) (-1)^{r} C(x_{\sigma_{r}},y_{r})^{-1} y_{r}
\end{align*}
where $y_{r_{s}}$ is defined by the canonical decomposition $x_{i_{r_{s}}} * u_{r} = u_{r_{s}}* y_{r_{s}}$ where $u_{r_{s}}=g(x_{i_{r_{s}}} * u_{r})$, for any $1\leq r\leq q$ and $1\leq s \leq q-1$. If $r>s$, the last two relations in \eqref{lem:ref1b} imply that 
\begin{align*}
    u_{r_{s}} &= g(x_{i_{r_{s}}} *u_{r})\\
    &= g( x_{i_{r_{s}}} *g(x_{i_{r}}* u))\\
    &= g( x_{i_{r_{s}}}* x_{i_{r}}* u) \textnormal{ by \Cref{rem:DecFun}(2)}\\
    &= g( x_{i_{s}} *x_{i_{s_{(r-1)}}} *u) \textnormal{ by \Cref{lem:ref1}}\\
    &= g( x_{i_{s_{(r-1)}}} *x_{i_{s}} *u)\\
    &= g( x_{i_{s_{(r-1)}}}* g(x_{i_{s}}* u)) \textnormal{ by \Cref{rem:DecFun}(2)}\\
    &= g( x_{i_{s_{(r-1)}}} *u_{s})\\
    &= u_{s_{(r-1)}}.
\end{align*}
By \eqref{lem:ref1b} and the above display, the basis elements $e(\sigma_{r_{s}};u_{r_{s}})$ and $e(\sigma_{s_{r-1}};u_{s_{r-1}})$ are equal. We claim that in this case, their scalar coefficients sum to 0. The element $e(\sigma_{r_{s}};u)$ occurs in $D_{2}^{2}(e(\sigma;u))$ with scalar coefficient
\begin{align*}
    &(-1)^{r+s} 
    C(x_{\sigma_{r_{s}}}, y_{r_{s}})^{-1} 
    C(x_{\sigma_{r}}, y_{r})^{-1} 
    C(y_{r_{s}},y_{r})\\
    = &(-1)^{r+s} 
    C(x_{\sigma_{r_{s}}}, x_{i_{r_{s}}})^{-1}
    C(x_{\sigma_{r_{s}}}, u_{r})^{-1}
    C(x_{\sigma_{r_{s}}}, u_{r_{s}})
    C(x_{\sigma_{r}}, x_{i_{r}})^{-1}
    C(x_{\sigma_{r}}, u)^{-1}
    C(x_{\sigma_{r}}, u_{r})\\
    &\cdot C(x_{i_{r_{s}}}, x_{i_{r}})
    C(x_{i_{r_{s}}}, u)
    C(x_{i_{r_{s}}}, u_{r})^{-1}
    \cancelto{1}{C(u_{r}, y_{r})}
    C(u_{r_{s}}, x_{i_{r}})^{-1}
    C(u_{r_{s}}, u)^{-1}
    C(u_{r_{s}}, u_{r})\\
    = &(-1)^{r+s} 
    C(x_{\sigma_{r_{s}}}, x_{i_{r_{s}}})^{-1}
    \cancel{C(x_{\sigma_{r_{s}}}, u_{r})^{-1}}
    C(x_{\sigma_{r_{s}}}, u_{r_{s}})
    C(x_{\sigma_{r}}, x_{i_{r}})^{-1}
    C(x_{\sigma_{r}}, u)^{-1}
    \cancel{C(x_{\sigma_{r}}, u_{r})}\\
    &\cdot C(x_{i_{r_{s}}}, x_{i_{r}})
    C(x_{i_{r_{s}}}, u)
    \cancel{C(x_{i_{r_{s}}}, u_{r})^{-1}}
    C(u_{r_{s}}, x_{i_{r}})^{-1}
    C(u_{r_{s}}, u)^{-1}
    C(u_{r_{s}}, u_{r})\\
    = &(-1)^{r+s} 
    C(x_{\sigma_{r_{s}}}, x_{i_{r_{s}}})^{-1}
    C(x_{\sigma_{r_{s}}}, u_{r_{s}})
    C(x_{\sigma_{r}}, x_{i_{r}})^{-1}
    C(x_{\sigma_{r}}, u)^{-1}\\
    &\cdot C(x_{i_{r_{s}}}, x_{i_{r}})
    C(x_{i_{r_{s}}}, u)
    C(u_{r_{s}}, x_{i_{r}})^{-1}
    C(u_{r_{s}}, u)^{-1}
    C(u_{r_{s}}, u_{r}).
\end{align*}
The element $e(\sigma_{s_{(r-1)}};u)$ occurs in $D_{2}^{2}(e(\sigma;u))$ with scalar coefficient
\begin{align*}
    &(-1)^{r+s-1} 
    C(x_{\sigma_{s_{(r-1)}}}, y_{s_{(r-1)}})^{-1} 
    C(x_{\sigma_{s}}, y_{s})^{-1} 
    C(y_{s_{(r-1)}},y_{s})\\
    = &(-1)^{r+s-1} 
    C(x_{\sigma_{s_{(r-1)}}}, x_{i_{s_{(r-1)}}})^{-1}
    C(x_{\sigma_{s_{(r-1)}}}, u_{s})^{-1}
    C(x_{\sigma_{s_{(r-1)}}}, u_{s_{(r-1)}})
    C(x_{\sigma_{s}}, x_{i_{s}})^{-1}
    C(x_{\sigma_{s}}, u)^{-1}
    C(x_{\sigma_{s}}, u_{s})\\
    &\cdot C(x_{i_{s_{(r-1)}}}, x_{i_{s}})
    C(x_{i_{s_{(r-1)}}}, u)
    C(x_{i_{s_{(r-1)}}}, u_{s})^{-1}
    \cancelto{1}{C(u_{s}, y_{s})}
    C(u_{s_{(r-1)}}, x_{i_{s}})^{-1}
    C(u_{s_{(r-1)}}, u)^{-1}
    C(u_{s_{(r-1)}}, u_{s})\\
    = &(-1)^{r+s-1} 
    C(x_{\sigma_{s_{(r-1)}}}, x_{i_{s_{(r-1)}}})^{-1}
    \cancel{C(x_{\sigma_{s_{(r-1)}}}, u_{s})^{-1}}
    C(x_{\sigma_{s_{(r-1)}}}, u_{s_{(r-1)}})
    C(x_{\sigma_{s}}, x_{i_{s}})^{-1}
    C(x_{\sigma_{s}}, u)^{-1}
    \cancel{C(x_{\sigma_{s}}, u_{s})}\\
    &\cdot C(x_{i_{s_{(r-1)}}}, x_{i_{s}})
    C(x_{i_{s_{(r-1)}}}, u)
    \cancel{C(x_{i_{s_{(r-1)}}}, u_{s})^{-1}}
    C(u_{s_{(r-1)}}, x_{i_{s}})^{-1}
    C(u_{s_{(r-1)}}, u)^{-1}
    C(u_{s_{(r-1)}}, u_{s})\\
    = &(-1)^{r+s-1} 
    C(x_{\sigma_{s_{(r-1)}}}, x_{i_{s_{(r-1)}}})^{-1}
    C(x_{\sigma_{s_{(r-1)}}}, u_{s_{(r-1)}})
    C(x_{\sigma_{s}}, x_{i_{s}})^{-1}
    C(x_{\sigma_{s}}, u)^{-1}\\
    &\cdot C(x_{i_{s_{(r-1)}}}, x_{i_{s}})
    C(x_{i_{s_{(r-1)}}}, u)
    C(u_{s_{(r-1)}}, x_{i_{s}})^{-1}
    C(u_{s_{(r-1)}}, u)^{-1}
    C(u_{s_{(r-1)}}, u_{s}).
\end{align*}
Clearly the signs will be opposite, so it suffices to show that 
\begin{align*}
    &C(x_{\sigma_{r_{s}}}, x_{i_{r_{s}}})^{-1}
    C(x_{\sigma_{r_{s}}}, u_{r_{s}})
    C(x_{\sigma_{r}}, x_{i_{r}})^{-1}
    C(x_{\sigma_{r}}, u)^{-1}\\
    &\cdot C(x_{i_{r_{s}}}, x_{i_{r}})
    C(x_{i_{r_{s}}}, u)
    C(u_{r_{s}}, x_{i_{r}})^{-1}
    C(u_{r_{s}}, u)^{-1}
    C(u_{r_{s}}, u_{r})
\end{align*}
is equal to 
\begin{align*}
    &C(x_{\sigma_{s_{(r-1)}}}, x_{i_{s_{(r-1)}}})^{-1}
    C(x_{\sigma_{s_{(r-1)}}}, u_{s_{(r-1)}})
    C(x_{\sigma_{s}}, x_{i_{s}})^{-1}
    C(x_{\sigma_{s}}, u)^{-1}\\
    &\cdot C(x_{i_{s_{(r-1)}}}, x_{i_{s}})
    C(x_{i_{s_{(r-1)}}}, u)
    C(u_{s_{(r-1)}}, x_{i_{s}})^{-1}
    C(u_{s_{(r-1)}}, u)^{-1}
    C(u_{s_{(r-1)}}, u_{s}).
\end{align*}

We notice that the following labeled factors
\begin{align*}
    &C(x_{\sigma_{r_{s}}}, x_{i_{r_{s}}})^{-1}
    \underbrace{C(x_{\sigma_{r_{s}}}, u_{r_{s}})}_{1}
    C(x_{\sigma_{r}}, x_{i_{r}})^{-1}
    C(x_{\sigma_{r}}, u)^{-1}\\
    &\cdot C(x_{i_{r_{s}}}, x_{i_{r}})
    C(x_{i_{r_{s}}}, u)
    C(u_{r_{s}}, x_{i_{r}})^{-1}
    \underbrace{C(u_{r_{s}}, u)^{-1}}_{2}
    C(u_{r_{s}}, u_{r})
\end{align*}
correspond to the following labeled factors 
\begin{align*}
    &C(x_{\sigma_{s_{(r-1)}}}, x_{i_{s_{(r-1)}}})^{-1}
    \underbrace{C(x_{\sigma_{s_{(r-1)}}}, u_{s_{(r-1)}})}_{1}
    C(x_{\sigma_{s}}, x_{i_{s}})^{-1}
    C(x_{\sigma_{s}}, u)^{-1}\\
    &\cdot C(x_{i_{s_{(r-1)}}}, x_{i_{s}})
    C(x_{i_{s_{(r-1)}}}, u)
    C(u_{s_{(r-1)}}, x_{i_{s}})^{-1}
    \underbrace{C(u_{s_{(r-1)}}, u)^{-1}}_{2}
    C(u_{s_{(r-1)}}, u_{s}).
\end{align*}
Thus it remains to show that
\begin{align*}
    C(x_{\sigma_{r_{s}}}, x_{i_{r_{s}}})^{-1}
    C(x_{\sigma_{r}}, x_{i_{r}})^{-1}
    C(x_{\sigma_{r}}, u)^{-1}
    C(x_{i_{r_{s}}}, x_{i_{r}})
    C(x_{i_{r_{s}}}, u)
    C(u_{r_{s}}, x_{i_{r}})^{-1}
    C(u_{r_{s}}, u_{r})
\end{align*}
is equal to 
\begin{align*}
    C(x_{\sigma_{s_{(r-1)}}}, x_{i_{s_{(r-1)}}})^{-1}
    C(x_{\sigma_{s}}, x_{i_{s}})^{-1}
    C(x_{\sigma_{s}}, u)^{-1}
    C(x_{i_{s_{(r-1)}}}, x_{i_{s}})
    C(x_{i_{s_{(r-1)}}}, u)
    C(u_{s_{(r-1)}}, x_{i_{s}})^{-1}
    C(u_{s_{(r-1)}}, u_{s}).
\end{align*}
In the computations below, we show that the following labeled factors
\begin{align*}
    \underbrace{C(x_{\sigma_{r_{s}}}, x_{i_{r_{s}}})^{-1}}_{6}
    \underbrace{C(x_{\sigma_{r}}, x_{i_{r}})^{-1}}_{5}
    \underbrace{C(x_{\sigma_{r}}, u)^{-1}}_{3}
    \underbrace{C(x_{i_{r_{s}}}, x_{i_{r}})}_{5}
    \underbrace{C(x_{i_{r_{s}}}, u)}_{3}
    \underbrace{C(u_{r_{s}}, x_{i_{r}})^{-1}
    C(u_{r_{s}}, u_{r})}_{4}
\end{align*}
cancel with the following labeled factors
\begin{align*}
    \underbrace{C(x_{\sigma_{s_{(r-1)}}}, x_{i_{s_{(r-1)}}})^{-1}}_{5}
    \underbrace{C(x_{\sigma_{s}}, x_{i_{s}})^{-1}}_{6}
    \underbrace{C(x_{\sigma_{s}}, u)^{-1}}_{3}
    \underbrace{C(x_{i_{s_{(r-1)}}}, x_{i_{s}})}_{6}
    \underbrace{C(x_{i_{s_{(r-1)}}}, u)}_{3}
    \underbrace{C(u_{s_{(r-1)}}, x_{i_{s}})^{-1}
    C(u_{s_{(r-1)}}, u_{s})}_{4}.
\end{align*}
Indeed, for the factors labeled with 3 one has: 
\begin{align*}
    1 &= C (1, u)\\
    &= C \left(\frac{x_{\sigma}}{x_{\sigma}}, u \right)\\
    &= C \left(\frac{x_{\sigma_{r}} * x_{i_{r}}}{x_{\sigma_{s}} * x_{i_{s}}}, u \right)\\
    &= C(x_{\sigma_{r}}, u) C(x_{i_{r}}, u) C(x_{\sigma_{s}},u)^{-1} C(x_{i_{_{s}}},u)^{-1}\\
    &= C(x_{\sigma_{r}}, u) C(x_{i_{s_{(r-1)}}}, u) C(x_{\sigma_{s}},u)^{-1} C(x_{i_{_{r_{s}}}},u)^{-1}.
\end{align*}
For the factors labeled with 4 one has: 
\begin{align*}
    1 &= C(u_{r_{s}}, y_{r_{s}}) C(u_{s_{(r-1)}}, y_{s_{(r-1)}}) \\
    &= C\left(u_{r_{s}}, \frac{y_{r_{s}}}{y_{s_{(r-1)}}}\right)\\
    &= C\left(u_{r_{s}}, \frac{u_{r_s}*y_{r_{s}}}{u_{s_{(r-1)}}*y_{s_{(r-1)}}}\right)\\
    &= C\left(u_{r_{s}}, \frac{x_{i_{{r_s}}}*u_{r}}{x_{i_{s_{(r-1)}}}*u_s}\right)\\
    &= C\left(u_{r_{s}},x_{i_{{r_s}}}\right)
    C\left(u_{r_{s}},u_{r}\right)
    C(u_{r_{s}},x_{i_{s_{(r-1)}}})^{-1}
    C\left(u_{r_{s}},u_s\right)^{-1}\\
    &= C\left(u_{s_{(r-1)}},x_{i_{s}}\right)
    C\left(u_{r_{s}},u_{r}\right)
    C\left(u_{r_{s}},x_{i_r}\right)^{-1}
    C\left(u_{s_{(r-1)}},u_s\right)^{-1}\\
\end{align*}
For the factors labeled with 5 one has:

\begin{align*}
    1 &= C (1, u_{s})\\
    &= C \left(\frac{x_{\sigma_{s}}}{x_{\sigma_{s}}}, u_{s} \right)\\
    &= C \left(\frac{x_{i_{s_{(r-1)}}}*x_{\sigma_{s_{(r-1)}}}}{x_{\sigma_{s}}} , u_{s} \right)\\
    &= C(x_{i_{s_{(r-1)}}}, u_{s})
    C(x_{\sigma_{s_{(r-1)}}},u_{s})
    C(x_{\sigma_{s}},u_{s})^{-1}
\end{align*}
For the factors labeled with 6, the computation follows similarly to the factors labeled with 5.
Thus $D=0$ and $(K_{\bullet}(I),D)$ is a complex. Now consider the submodule $J_{q} \subset K_{q}$ of admissible symbols satisfying $i_{q}\geq\max(u)$. We claim that $(J_{\bullet}(I),D)$ is a subcomplex of $(K_{\bullet}(I),D)$. For $e(\sigma;u)\in J_{q}$, consider the $q^{th}$ summand of $D(e(\sigma;u))$, i.e. 
\[ (-1)^q e(\sigma_{q};u) C(x_{\sigma_{q}}*u, x_{i_{q}})^{-1} x_{i_{q}} - (-1)^q e(\sigma_{q};u_{q}) C(x_{\sigma_{q}}, y_{q})^{-1} y_{q}. \]
By \Cref{rem:DecFun}(1) $i_{q} \geq \max(u)$ implies that $u_{q} = u$ and thus $y_{q} = x_{i_{q}}$. So we can rewrite this summand as 
\begin{align*}
&(-1)^q e(\sigma_{q};u) C(x_{\sigma_{q}}*u, x_{i_{q}})^{-1} x_{i_{q}} - 
(-1)^q e(\sigma_{q};u) C(x_{\sigma_{q}}, x_{i_{q}})^{-1} x_{i_{q}}\\
= &(-1)^q e(\sigma_{q};u) C(x_{\sigma_{q}}, x_{i_{q}})^{-1} C(u, x_{i_{q}})^{-1} x_{i_{q}} - (-1)^q e(\sigma_{q};u) C(x_{\sigma_{q}}, x_{i_{q}})^{-1} x_{i_{q}}\\
= &(-1)^q e(\sigma_{q};u) C(x_{\sigma_{q}}, x_{i_{q}})^{-1} \cancelto{1}{C(u, x_{i_{q}})^{-1}} x_{i_{q}} - (-1)^q e(\sigma_{q};u) C(x_{\sigma_{q}}, x_{i_{q}})^{-1} x_{i_{q}}\\
= &(-1)^q e(\sigma_{q};u) C(x_{\sigma_{q}}, x_{i_{q}})^{-1} x_{i_{q}} - (-1)^q e(\sigma_{q};u) C(x_{\sigma_{q}}, x_{i_{q}})^{-1} x_{i_{q}}\\
= &0.
\end{align*} 
Thus $D(e(\sigma;u))$ can be written as 
\[ D(e(\sigma;u)) := \sum_{r=1}^{q-1} (-1)^{r} e(\sigma_{r};u) C(x_{\sigma_{r}}*u,x_{i_{r}})^{-1} x_{i_{r}} - \sum_{r=1}^{q-1} (-1)^{r} e(\sigma_{r};u_{r}) C(x_{\sigma_{r}},y_{r})^{-1} y_{r}. \]
By \Cref{rem:DecFun}(2), $\max(u_r) = \max(g(x_{i_{r}}*u)) \leq \max(g(u)) = \max(u) \leq i_{q}$, so $D(e(\sigma;u)) \in J_{q-1}$. Clearly $(L_{\bullet}(I),d) = \faktor{(K_{\bullet}(I),D)}{(J_{\bullet}(I),D)}$ and the differential $d:L_{q}\rightarrow L_{q-1}$ is induced by the differential $D:K_{q}\rightarrow K_{q-1}$. Hence $d^2=0$ and $(L_{\bullet}(I),d)$ is a complex.
\end{proof}

The proof of the following Theorem follows as in the commutative case and it is therefore omitted, see \cite[Theorem 2.1]{EK}.

\begin{theorem}
The complex $(L_\bullet(I),d)$ is a free resolution of $I$ over $R$. Moreover, if $k$ is a field then $(L_{\bullet}(I),d)$ is a minimal resolution of $I$.
\end{theorem}

\begin{definition}
The \textit{skew Eliahou-Kervaire resolution} of a stable ideal $I$ is the complex $(L_{\bullet}(I),d)$.
\end{definition}

\begin{example}\label{ex:Example}
Consider the ideal $I = (x^2, xy, y^2 )$ in the skew polynomial ring $R=k_q[x,y]$, where $xy=qyx$ for some $q\in k^*$ and where $k$ is a field. The skew Eliahou-Kervaire resolution of this ideal is given by 
\[ 0 \rightarrow e(1;xy)R \oplus e(1;y^2)R \rightarrow e(\emptyset;x^2)R \oplus e(\emptyset;xy)R \oplus e(\emptyset;y^2)R \rightarrow 0, \]
where the differential is represented by the matrix 
\[
\begin{pmatrix}
y & 0\\
-qx & y\\
0 & -q^{2}x
\end{pmatrix}.
\]
We point out that the skew Taylor resolution, defined in \cite{Taylor}, of this ideal is not minimal.
\end{example}

\section{Applications}
In this section we list several applications that were already known in the commutative case. 

The proof of the following corollary follows as in the commutative case, see \cite[Section 3]{EK}, and therefore is omitted.

\begin{corollary}
Let $I$ be a stable ideal in a skew polynomial ring $R$ with $n$ variables. Then
\begin{enumerate}
\item The Betti numbers of $I$ are given by 
\[
\beta_q(I)=\sum_{u\in G(I)}\binom{\max(u)-1}{q}.
\]
\item The projective dimension of $I$ is given by
\[
\mathrm{pd}_RI=\max\{\max(u)-1\mid u\in G(I)\}.
\]
\item The Poincar\'{e} series of $I$ is given by
\[
\mathrm{P}_I^R(t)=\sum_{u\in G(I)}\frac{t^{\mathrm{deg}(u)}}{(1-t)^{n-\max(u)+1}}.
\]
\end{enumerate}
\end{corollary}

\begin{corollary}
Let $\m$ be the ideal $(x_1,\ldots, x_n)$ in $R=k_\mathfrak{q}[x_1,\ldots,x_n]$. Then, for any $d\geq1$, the Betti numbers of $\m^d$ are given by
\[
\beta_q(\m^d)=\binom{d+n-1}{d+q}\binom{d+q-1}{q}.
\]
\end{corollary}

\begin{proof}
This corollary is already known in the commutative case, see \cite[Example 1]{EK}. Since the ranks of the free modules in the skew Eliahou-Kervaire resolution of $\m^d$ are the same as the ranks of the free modules in the Eliahou-Kervaire resolution of the ideal $(y_1,\ldots,y_n)^d$ of the commutative ring $k[y_1,\ldots,y_n]$, the corollary follows from the commutative case.
\end{proof}

Arguing as in the proof of the previous corollary we can also deduce the following corollary from \cite[Example 2]{EK}.

\begin{corollary}
Let $S_n$ be the ideal in $k_\mathfrak{q}[x_1,\ldots,x_n]$ generated by the set of monomials $w$ such that $\mathrm{deg}(w)=\max(w)$. Then, the Betti numbers of $S_n$ are given by
\[
\beta_q(S_n)=\sum_{m=1}^n\frac{1}{m}\binom{2m-2}{m-1}\binom{m-1}{q}.
\]
\end{corollary}

\begin{chunk}
Let $M$ be a module over the skew polynomial ring $R=k_\mathfrak{q}[x_1,\ldots,x_n]$. Since $R$ is $\mathbb{Z}$-graded then, by keeping track of the degrees in the minimal free resolution of $M$, one can define the \emph{Tor-regularity} of $M$ as
\[
\mathrm{Tor.reg}_RM=\mathrm{sup}\{j-i\mid \beta_{i,j}(M)\neq0\},
\]
where $\beta_{i,j}(M)$ is the rank of the free module in internal degree $j$ and homological degree $i$ in the minimal free resolution of $M$.

One can also define the \emph{Castelnuovo-Mumford regularity} of $M$, as in \cite[Definition 2.1]{Jo4}. We denote this invariant by $\mathrm{CM.reg}_RM$.

\end{chunk}

As for the commutative case, the following corollary follows directly from the construction of the skew Eliahou-Kervaire resolution of a stable ideal.
\begin{corollary}\label{cor:TorReg}
Let $I$ be a stable ideal in a skew polynomial ring $R$. Then the Tor-regularity of $I$ is given by
\[
\mathrm{Tor.reg}_RI=\max\{\mathrm{deg}(u)\mid u\in G(I)\}.
\]
\end{corollary}

As a final application we present the following
\begin{corollary}
Let $I$ be a stable ideal in a skew polynomial ring $R$. Then the Castelnuovo-Mumford regularity of $I$ is given by
\[
\mathrm{CM.reg}_RI=\max\{\mathrm{deg}(u)\mid u\in G(I)\}.
\]
\end{corollary}

\begin{proof}
It follows from \cite[Theorem 5.4]{Dong} and \cite[Corollary 4.14]{Yek} that over a skew polynomial ring $R$, the Castelnuovo-Mumford regularity of $I$ and the Tor-regularity of $I$ coincide. Now one concludes by invoking \cref{cor:TorReg}.
\end{proof}

\section{A product on the skew Elihaou-Kervaire Resolution}

In \cite[Remark 2]{EK}, Eliahou and Kervaire show that the resolution they construct admits a nonunital, associative, graded commutative product satisfying the Leibniz rule. A unital, associative, graded commutative product was first defined by Srinivasan, for powers of the maximal irrelevant ideal, see \cite{srinivasan1989algebra}, and later by Peeva, for stable ideals, see \cite{Peeva}. In this section we extend Eliahou and Kervaire's result to our noncommutative context.

\begin{chunk}
Color differential graded (DG) algebras have been introduced in \cite[Definition 4.1]{FM} by the first author and W. F. Moore. They have been used to study the homological properties of quotients of skew polynomial rings by ideals generated by homogeneous normal elements. A DG algebra $A$ over the skew polynomial ring $R$ is a bigraded unital associative $k$-algebra with $R\subseteq A_0$ equipped with a graded $R$-linear differential $\partial$ of homological degree $-1$ such that $\partial^2=0$, and such that the Leibniz rule holds:
\[
\partial(ab)=\partial(a)b+(-1)^{|a|}a\partial(b),
\]
where $a$ and $b$ are bihomogeneous elements.

The DG algebra $A$ is said to be a color DG algebra if $A$ has a $G$-grading compatible with the bigrading of $A$ and such that $\partial$ is $G$-homogeneous of $G$-degree $e_G$.

We say that a color DG algebra $A$ is graded color commutative if for all trihomogeneous elements $x,y\in A$ one has
\[
xy=(-1)^{|x||y|}\chi(x,y)yx,
\]
where $\chi(x,y)$ is the usual abuse of notation, and $x^2=0$ if the homological degree of $x$ is odd.
\end{chunk}

\begin{chunk}\label{chk:DGA}
We define the following product on the resolution $(L_{\bullet}(I),d)$: if $\sigma$ and $\tau$ share an index, then 
\[
e(\sigma;u) e(\tau;v) = 0,
\]
otherwise $e(\sigma;u) e(\tau;v)$ is equal to
\[
(-1)^{\inv(\sigma,\tau)}e(\sigma * \tau; g(u*v)) \chi(u,x_{\tau}) C(x_{\sigma}, x_{\tau}) C(u,v) C\left(x_{\sigma}, \frac{g(u*v)}{u}\right) C\left(x_{\tau}, \frac{g(u*v)}{v}\right) \frac{u*v}{g(u*v)} 
\]
where $\inv(\sigma,\tau)=|\{(i,j)\in\sigma\times\tau\mid j<i\}|$.
\end{chunk}

\begin{theorem}\label{thm:DGA}
The product defined in \cref{chk:DGA} is associative, graded color commutative and satisfies the Leibniz rule.
\end{theorem}

\begin{remark}
The skew Eliahou-Kervaire resolution, with the product defined in \cref{chk:DGA}, is not a color DG algebra since it does not admit a unit.
\end{remark}

We divide the proof of \cref{thm:DGA} into different lemmas.

\begin{lemma}\label{lem:ColorComm}
The product defined in \Cref{chk:DGA} is graded color commutative, where the $G$-degree of $e(\sigma,u)$ is the same as $x_\sigma*u$.
\end{lemma}

\begin{proof}
One needs to show that 
\begin{equation}\label{eq:ColorComm}
e(\sigma;u)e(\tau;v) = (-1)^{l_\sigma + l_\tau} \chi(x_\sigma * u, x_\tau * v)e(\tau;v)e(\sigma;u),
\end{equation}
where $l_\sigma$ is the length of $\sigma$ and $l_\tau$ is the length of $\tau$. Ignoring the signs, which are the same by the commutative case, this will follow provided
\begin{align*}
&C(x_{\sigma}*u,x_{\tau}*v) C\left(x_{\sigma}*x_{\tau}, \frac{u*v}{g(u*v)}\right)^{-1} \frac{u*v}{g(u*v)}\\
= &\chi(x_\sigma * u, x_\tau * v) C(x_{\tau}*v,x_{\sigma}*u) C\left(x_{\sigma}*x_{\tau}, \frac{u*v}{g(u*v)}\right)^{-1} \frac{u*v}{g(u*v)},
\end{align*}
which follows by repeated applications of \cref{eq:CChiRel}.

We now show that the square of a trihomogeneous element of odd homological degree is 0. Let 
\begin{equation}\label{eq:OddElement}
\sum_{i=1}^m p_ie(\sigma_i,u_i)
\end{equation}
be such an element, with $p_i\in R$. Being trihomogeneous implies that all the summands $p_ie(\sigma_i,u_i)$ have the same $G$-degree and therefore
\begin{equation}\label{eq:chi=1}
\chi(p_ie(\sigma_i,u_i),p_je(\sigma_j,u_j))=1,\quad\forall i,j.
\end{equation}
Since \cref{eq:OddElement} is odd, we deduce that $l_{\sigma_i}$ is odd for all $i$. Finally
\begin{equation}\label{eq:DoubleSum}
\left(\sum_{i=1}^m p_ie(\sigma_i,u_i)\right)^2=\sum_{i=1}^m\sum_{j=1}^mp_ie(\sigma_i,u_i)p_je(\sigma_j,u_j),
\end{equation}
by \cref{eq:ColorComm} and \cref{eq:chi=1} one has
\[
p_ie(\sigma_i,u_i)p_je(\sigma_j,u_j)=-p_je(\sigma_j,u_j)p_ie(\sigma_i,u_i),
\]
therefore, if the characteristic of $k$ is not 2, then the elements of \eqref{eq:DoubleSum} will cancel because each one appears twice but with opposite signs. If the characteristic of $k$ is 2, then each element of \eqref{eq:DoubleSum} appears twice, picking up 2 as a coefficient, giving zero as a result. This shows that the square is zero in any characteristic.
\end{proof}

\begin{lemma}
The product defined in \Cref{chk:DGA} is associative.
\end{lemma}

\begin{proof}
As in the commutative case, the signs and the monomial coefficients will coincide, therefore in the remaining computations they will be ignored. If we expand $ \left( e(\sigma;u) e(\tau;v) \right) e(\rho;w)$, the basis element $e(\sigma * \tau *\rho; g(g(u*v)*w))$ occurs with scalar coefficient
\begin{align*}
    &C(x_{\sigma}*u,x_{\tau}*v)
    C\left( x_{\sigma}*x_{\tau},\frac{u*v}{g(u*v)} \right)^{-1}
    C(x_{\sigma}*x_{\tau}*g(u*v),x_{\rho}*w)
    \chi\left(\frac{u*v}{g(u*v)},x_{\rho}*w\right)\\
    &\cdot C\left( x_{\sigma}*x_{\tau}*x_{\rho}, \frac{g(u*v)*w}{g(g(u*v)*w)} \right)^{-1}
    C\left( \frac{g(u*v)*w}{g(g(u*v)*w)}, \frac{u*v}{g(u*v)} \right).
\end{align*}
After expanding using the property that $C$ and $\chi$ are bicharacters, we see that the following labeled factors cancel to 1
\begin{align*}
    &C(x_{\sigma},x_{\tau})
   \underbrace{C(x_{\sigma},v)}_{1}
    C(u,x_{\tau})
    C(u,v)
    C(x_{\sigma},u)^{-1}
    \underbrace{C(x_{\sigma},v)^{-1}}_{1}
    \underbrace{C(x_{\sigma},g(u*v))}_{2}
    C(x_{\tau},u)^{-1}
    C(x_{\tau},v)^{-1}\\
    &\cdot \underbrace{C(x_{\tau},g(u*v))}_{3}
    C(x_{\sigma},x_{\rho})
    \underbrace{C(x_{\sigma},w)}_{4}
    C(x_{\tau},x_{\rho})
    \underbrace{C(x_{\tau},w)}_{5}
    \underbrace{C(g(u*v),x_{\rho})}_{6}
    \underbrace{C(g(u*v),w)}_{7}
    C(u,x_{\rho})
    C(u,w)\\
    &\cdot C(v,x_{\rho})
     C(v,w)
    \underbrace{C(g(u*v),x_{\rho})^{-1}}_{6}
    \underbrace{C(g(u*v),w)^{-1}}_{7}
    C(x_{\rho},u)^{-1}
    \underbrace{C(w,u)^{-1}}_{8}
    C(x_{\rho},v)^{-1}
    \underbrace{C(w,v)^{-1}}_{9}\\
    &\cdot\underbrace{C(x_{\rho},g(u*v))}_{10}
     \underbrace{C(w,g(u*v))}_{11}
    \underbrace{C(x_{\sigma},g(u*v))^{-1}}_{2}
    \underbrace{C(x_{\sigma},w)^{-1}}_{4}
    C(x_{\sigma},g(g(u*v)*w))
    \underbrace{C(x_{\tau},g(u*v))^{-1}}_{3}\\
    &\cdot\underbrace{C(x_{\tau},w)^{-1}}_{5}
    C(x_{\tau},g(g(u*v)*w))
    \underbrace{C(x_{\rho},g(u*v))^{-1}}_{10}
    C(x_{\rho},w)^{-1}
    C(x_{\rho},g(g(u*v)*w))
    \cancelto{1}{C\left(g(u*v),\frac{u*v}{g(u*v)}\right)}\\
    &\cdot\underbrace{C(w,u)}_{8}
    \underbrace{C(w,v)}_{9}
    \underbrace{C(w,g(u*v))^{-1}}_{11}
    C(g(g(u*v)*w),u)^{-1}
    C(g(g(u*v)*w),v)^{-1}
    C(g(g(u*v)*w),g(u*v))
\end{align*}
and this expression is equal to:
\begin{align*}
    &C(x_{\sigma},x_{\tau})
    C(u,x_{\tau})
    C(u,v)
    C(x_{\sigma},u)^{-1}
    C(x_{\tau},u)^{-1}
    C(x_{\tau},v)^{-1}
    C(x_{\sigma},x_{\rho})
    C(x_{\tau},x_{\rho})
    C(u,x_{\rho})
    C(u,w)
    C(v,x_{\rho})\\
    &\cdot C(v,w)
    C(x_{\rho},u)^{-1}
    C(x_{\rho},v)^{-1}
    C(x_{\sigma},g(g(u*v)*w))
    C(x_{\tau},g(g(u*v)*w))
    C(x_{\rho},w)^{-1}
    C(x_{\rho},g(g(u*v)*w))\\
    &\cdot C(g(g(u*v)*w),u)^{-1}
    C(g(g(u*v)*w),v)^{-1}
    C(g(g(u*v)*w),g(u*v)).
\end{align*}
Similarly, if we expand $e(\sigma;u) \left( e(\tau;v) e(\rho;w) \right) $, the basis element $e(\sigma * \tau *\rho; g(u*g(v*w)))$ occurs with the following scalar coefficient
\begin{align*}
    &C(x_{\tau},x_{\rho})
    C(v,x_{\rho})
    C(v,w)
    C(x_{\tau},v)^{-1}
    C(x_{\rho},v)^{-1}
    C(x_{\rho},w)^{-1}
    C(x_{\sigma},x_{\tau})
    C(x_{\sigma},x_{\rho})
    C(u,x_{\tau})
    C(u,x_{\rho})
    C(x_{\sigma},u)^{-1}\\
    &\cdot C(x_{\sigma},g(u*g(v*w))
    C(x_{\tau},u)^{-1}
    C(x_{\tau},g(u*g(v*w)))
    C(x_{\rho},u)^{-1}
    C(x_{\rho},g(u*g(v*w)))
    C(u,v)
    C(u,w)\\
    &\cdot C(g(u*g(v*w)),v)^{-1}
    C(g(u*g(v*w)),w)^{-1}
    C(g(u*g(v*w)),g(v*w)).
\end{align*}
We notice that the following labeled factors in the first coefficient
\begin{align*}
    &\underbrace{C(x_{\sigma},x_{\tau})}_{12}
    \underbrace{C(u,x_{\tau})}_{13}
    \underbrace{C(u,v)}_{14}
    \underbrace{C(x_{\sigma},u)^{-1}}_{15}
    \underbrace{C(x_{\tau},u)^{-1}}_{16}
    \underbrace{C(x_{\tau},v)^{-1}}_{17}
    \underbrace{C(x_{\sigma},x_{\rho})}_{18}
    \underbrace{C(x_{\tau},x_{\rho})}_{19}
    \underbrace{C(u,x_{\rho})}_{20}
    \underbrace{C(u,w)}_{21}
    \underbrace{C(v,x_{\rho})}_{22}\\
    &\cdot\underbrace{C(v,w)}_{23}
    \underbrace{C(x_{\rho},u)^{-1}}_{24}
    \underbrace{C(x_{\rho},v)^{-1}}_{25}
    \underbrace{C(x_{\sigma},g(g(u*v)*w))}_{26}
    \underbrace{C(x_{\tau},g(g(u*v)*w))}_{27}
    \underbrace{C(x_{\rho},w)^{-1}}_{28}\\
    &\cdot\underbrace{C(x_{\rho},g(g(u*v)*w))}_{29}
    C(g(g(u*v)*w),u)^{-1}
    \underbrace{C(g(g(u*v)*w),v)^{-1}}_{30}
    C(g(g(u*v)*w),g(u*v))
\end{align*}
correspond to the following labeled factors in the second coefficient
\begin{align*}
    &\underbrace{C(x_{\tau},x_{\rho})}_{19}
    \underbrace{C(v,x_{\rho})}_{22}
    \underbrace{C(v,w)}_{23}
    \underbrace{C(x_{\tau},v)^{-1}}_{17}
    \underbrace{C(x_{\rho},v)^{-1}}_{25}
    \underbrace{C(x_{\rho},w)^{-1}}_{28}
    \underbrace{C(x_{\sigma},x_{\tau})}_{12}
    \underbrace{C(x_{\sigma},x_{\rho})}_{18}
    \underbrace{C(u,x_{\tau})}_{13}
    \underbrace{C(u,x_{\rho})}_{20}\\
    &\cdot \underbrace{C(x_{\sigma},u)^{-1}}_{15}
    \underbrace{C(x_{\sigma},g(u*g(v*w))}_{26}
    \underbrace{C(x_{\tau},u)^{-1}}_{16}
    \underbrace{C(x_{\tau},g(u*g(v*w)))}_{27}
    \underbrace{C(x_{\rho},u)^{-1}}_{24}
    \underbrace{C(x_{\rho},g(u*g(v*w)))}_{29}\\
    &\cdot \underbrace{C(u,v)}_{14}
    \underbrace{C(u,w)}_{21}
    \underbrace{C(g(u*g(v*w)),v)^{-1}}_{30}
    C(g(u*g(v*w)),w)^{-1}
    C(g(u*g(v*w)),g(v*w)).
\end{align*}
Thus it suffices to show that
\[
C(g(g(u*v)*w),u)^{-1} C(g(g(u*v)*w),g(u*v)) = C(g(u*g(v*w)),w)^{-1} C(g(u*g(v*w)),g(v*w)),  
\]
which follows via the computation
\begin{align*}
    1 &= C\left(g(u*g(v*w)), \frac{u*g(v*w)}{g(u*g(v*w))}\right)
    C\left(g(w*g(u*v)), \frac{w*g(u*v)}{g(w*g(u*v))}\right)^{-1}\textnormal{ by \Cref{rem:C=1}}\\
    &= C\left(g(u*v*w), \frac{u*g(v*w)}{g(u*v*w)}\right)
    C\left(g(u*v*w), \frac{w*g(u*v)}{g(u*v*w)}\right)^{-1} \textnormal{ by \Cref{rem:DecFun}(2)}\\
    &= C\left(g(u*v*w), \frac{\frac{u*g(v*w)}{g(u*v*w)}}{\frac{w*g(u*v)}{g(u*v*w)}}\right)\\
    &= C\left(g(u*v*w), \frac{u*g(v*w)}{w*g(u*v)}\right)\\
    &= C(g(u*v*w), u)
    C(g(u*v*w), g(v*w))
    C(g(u*v*w), w)^{-1}
    C(g(u*v*w), g(u*v))^{-1}\\
    &= C(g(g(u*v)*w), u)
    C(g(u*g(v*w)), g(v*w))
    C(g(u*g(v*w)), w)^{-1}
    C(g(g(u*v)*w), g(u*v))^{-1}.
\end{align*}
Finally, we notice that the basis elements $e(\sigma * \tau *\rho; g(g(u*v)*w))$ and $e(\sigma * \tau *\rho; g(u*g(v*w)))$ are the same by \Cref{rem:DecFun}(2).
\end{proof}

\begin{lemma}
The product defined in \Cref{chk:DGA} satisfies the Leibniz rule.
\end{lemma}

\begin{proof} We want to show that
\begin{align*}
    d(e(\sigma;u)e(\tau;v)) = d(e(\sigma;u))e(\tau;v) + (-1)^{|e(\sigma;u)|} e(\sigma;u)d(e(\tau;v)).
\end{align*}
We split the proof into two cases.

\noindent
\textbf{Case 1}: We assume that $\sigma$ and $\tau$ have an index in common.

We want to show that $d(e(\sigma;u))e(\tau;v) + (-1)^{|e(\sigma;u)|} e(\sigma;u)d(e(\tau;v)) = 0$. If $\sigma$ and $\tau$ have more than one index in common, then what we want to show is clear. Suppose $\sigma = (i_{1}, i_{2}, \dots i_{p})$, $\tau = (j_{1}, j_{2}, \dots j_{q})$ and $\sigma$ and $\tau$ have exactly one index in common. Then $i_{r} = j_{s}$ for some $r$ and $s$. In this case, the only nontrivial product in $d(e(\sigma;u))e(\tau;v)$ occurs when $i_r$ is removed from the lefthand factor, resulting in the product
\begin{align*}
    &\left[
    e(\sigma_{r};u)
    (-1)^{r}
    C(x_{\sigma_{r}}*u,x_{i_{r}})^{-1}
    x_{i_{r}}
    -
    e(\sigma_{r};u_{r})
    (-1)^{r}
    C(x_{\sigma_{r}}, y_{r})^{-1}
    y_{r}
    \right]
    e(\tau;v),
\end{align*}
which, by \Cref{lem:ColorComm}, is equal to
\begin{align*}
    &e(\sigma_{r};u)
    e(\tau;v)
    (-1)^{r}
    \chi(x_{i_{r}},x_{\tau}*v)
    C(x_{\sigma_{r}}*u,x_{i_{r}})^{-1}
    x_{i_{r}}
    -
    e(\sigma_{r};u_{r})
    e(\tau;v)
    (-1)^{r}
    \chi(y_{r},x_{\tau}*v)
    C(x_{\sigma_{r}}, y_{r})^{-1}
    y_{r}.
\end{align*}
Using the formula for the product given in \Cref{chk:DGA} we see that this last expression is equal to
\begin{align}\label{eq:Prod1}
    &(-1)^{r+\inv(\sigma_r, \tau)}
    e(\sigma_{r}*\tau;g(u*v))
    C(x_{\sigma_r}*u,x_{\tau}*v)
    C\left( x_{\sigma_r}*x_{\tau},\frac{u*v}{g(u*v)} \right)^{-1}\\
    &\cdot \chi(x_{i_{r}},x_{\tau}*v)
    C(x_{\sigma_{r}}*u,x_{i_{r}})^{-1}
    C\left(\frac{u*v}{g(u*v)},x_{i_{r}}\right)
    \frac{u*v}{g(u*v)}*x_{i_{r}} \nonumber\\
    - &(-1)^{r+\inv(\sigma_r,\tau)}
    e(\sigma_{r}*\tau;g(u_{r}*v))
    C(x_{\sigma_r}*u_r,x_{\tau}*v) C\left(x_{\sigma_r}*x_{\tau},\frac{u_r*v}{g(u_r*v)}\right)^{-1}\nonumber\\
    &\cdot \chi(y_{r},x_{\tau}*v)
    C(x_{\sigma_{r}}, y_{r})^{-1}
    C\left(\frac{u_{r}*v}{g(u_{r}*v)}, y_{r} \right)
    \frac{u_{r}*v}{g(u_{r}*v)}*y_{r}.\nonumber
\end{align}
Likewise, the only nontrivial product in $e(\sigma;u)d(e(\tau;v))$ occurs when $j_s$ is removed from the righthand factor. Proceeding as before we obtain
\begin{align}\label{eq:Prod2}
    &(-1)^{s+\inv(\sigma,\tau_s)} 
    e(\sigma*\tau_{s};g(u*v))
    C(x_{\sigma}*u,x_{\tau_s}*v)
    C\left( x_{\sigma}*x_{\tau_s},\frac{u*v}{g(u*v)} \right)^{-1}\\
    &\cdot C(x_{\tau_{s}}*v,x_{j_s})^{-1}
    C\left(\frac{u*v}{g(u*v)}, x_{j_s}\right)
    \frac{u*v}{g(u*v)}*x_{j_s} \nonumber\\
    - &(-1)^{s+\inv(\sigma,\tau_s)}
    e(\sigma*\tau_{s};g(u*v_{s}))
    C(x_{\sigma}*u,x_{\tau_s}*v_s) C\left(x_{\sigma}*x_{\tau_s},\frac{u*v_s}{g(u*v_s)}\right)^{-1}\nonumber\\
    &\cdot C(x_{\tau_{s}}, y_{s})^{-1}
    C\left(\frac{u*v_{s}}{g(u*v_{s})}, y_{s}\right)
    \frac{u*v_{s}}{g(u*v_{s})}*y_{s}.\nonumber
\end{align}
Notice that $e(\sigma_{r}*\tau;g(u_r*v)) =  e(\sigma*\tau_{s};g(u*v_s))$, since $i_{r}=j_{s}$ implies that $x_{i_{r}} = x_{j_s}$, $\sigma_{r}*\tau = \sigma*\tau_{s}$, and 
\begin{align*}
    g(u_{r}*v) = g(g(x_{i_{r}}*u)*v) = g(x_{i_{r}}*u*v) = g(x_{j_s}*u*v) = g(u*x_{j_s}*v) = g(u*g(x_{j_s}*v)) = g(u*v_{s}).
\end{align*}
As in the commutative case, see \cite[Theorem 2.1]{EK}, it follows that the signs coincide, therefore they will be ignored in the following computations. We start by showing that the scalar coefficient of $e(\sigma_{r}*\tau;g(u_{r}*v))$ in \cref{eq:Prod1} coincides with the scalar coefficient of $e(\sigma*\tau_{s};g(u*v_{s}))$ in \cref{eq:Prod2}. After expanding the scalar coefficient of $e(\sigma_{r}*\tau;g(u_{r}*v))$ using the property that $C$ and $\chi$ are bicharacters and rewriting $x_{\sigma_r} = \frac{x_{\sigma}}{x_{i_r}}$ and $y_{r} = \frac{x_{i_{r}}*u}{u_{r}}$, we see that the following labeled factors cancel to 1: 
\begin{align*}
    & C(x_{\sigma}x_{\tau}) 
    \underbrace{C(x_{i_r},x_{\tau})^{-1}}_{1} 
    \underbrace{C(u_r,x_{\tau})}_{2} 
    \underbrace{C(x_{\sigma},v)}_{3} 
    \underbrace{C(x_{i_r},v)^{-1}}_{4} 
    \underbrace{C(u_r,v)}_{5} 
    C(x_{\sigma},g(u_r*v)) 
    C(x_{i_r},g(u_r*v))^{-1}\\
    &\cdot \underbrace{C(x_{\sigma},u_r)^{-1}}_{6}
    \underbrace{C(x_{i_r},u_r)}_{7}
    \underbrace{C(x_{\sigma},v)^{-1}}_{3}
    \underbrace{C(x_{i_r},v)}_{4}
    C(x_{\tau},g(u_r*v))
    \underbrace{C(x_{\tau},u_r)^{-1}}_{8}
    C(x_{\tau},v)^{-1}
    \underbrace{C(x_{i_r},x_{\tau})}_{1}
    C(u,x_{\tau})\\
    &\cdot \underbrace{C(u_r,x_{\tau})^{-1}}_{2}
    C(x_{i_r},v)
    C(u,v)
    \underbrace{C(u_r,v)^{-1}}_{5}
    C(x_{\tau},x_{i_r})^{-1}
    C(x_{\tau},u)^{-1}
    \underbrace{C(x_{\tau},u_r)}_{8}
    \underbrace{C(v,y_r)^{-1}}_{9}
    \cancelto{1}{C(u_r,y_r)}\\
    &\cdot C(x_{\sigma},x_{i_r})^{-1}
    C(x_{\sigma},u)^{-1}
    \underbrace{C(x_{\sigma},u_r)}_{6}
    \cancelto{1}{C(x_{i_r},x_{i_r})}
    C(x_{i_r},u)
    \underbrace{C(x_{i_r},u_r)^{-1}}_{7}
    \cancelto{1}{C(u_r,y_r)}
    \underbrace{C(v,y_r)}_{9}\\
    &\cdot C(g(u_r*v),x_{i_r})^{-1}
    C(g(u_r*v),u)^{-1}
    C(g(u_r*v),u_r)
\end{align*}
and this expression is equal to:
\small
\begin{align}\label{eq:Prod3}
    & C(x_{\sigma}x_{\tau}) 
    C(x_{\sigma},g(u_r*v)) 
    C(x_{i_r},g(u_r*v))^{-1}
    C(x_{\tau},g(u_r*v))
    C(x_{\tau},v)^{-1}
    C(u,x_{\tau})
    C(x_{i_r},v)
    C(u,v)
    C(x_{\tau},x_{i_r})^{-1}\\
    &\cdot C(x_{\tau},u)^{-1}
    C(x_{\sigma},x_{i_r})^{-1}
    C(x_{\sigma},u)^{-1}
    C(x_{i_r},u)
    C(g(u_r*v),x_{i_r})^{-1}
    C(g(u_r*v),u)^{-1}
    C(g(u_r*v),u_r)\nonumber.
\end{align}
\normalsize
Likewise, rewriting $x_{\tau_s}=\frac{x_{\tau}}{x_{j_s}}$ and $y_s=\frac{x_{j_s}*v}{v_s}$, it follows that the scalar coefficient of $e(\sigma*\tau_{s};g(u*v_{s}))$ is
\small
\begin{align}\label{eq:Prod4}
    &C(x_{\sigma},x_{\tau})
    C(x_{\sigma},x_{i_r})^{-1}
    C(u,x_{\tau})
    C(x_{\sigma},g(u*v_s))
    C(x_{\sigma},u)^{-1}
    C(x_{\tau},g(u*v_s))
    C(x_{j_s},g(u*v_s))^{-1}
    C(x_{\tau},u)^{-1}\\
    &\cdot C(x_{j_s},u)
    C(x_{\tau},x_{j_s})^{-1}
    C(x_{\tau},v)^{-1}
    C(x_{j_s},v)
    C(u,v)
    C(g(u*v_s),x_{j_s})^{-1}
    C(g(u*v_s),v)^{-1}
    C(g(u*v_s),v_s).\nonumber
\end{align}
\normalsize
We notice that the following labeled factors in \eqref{eq:Prod3}
\small
\begin{align*}
    & \underbrace{C(x_{\sigma}x_{\tau})}_{10}
    \underbrace{C(x_{\sigma},g(u_r*v)) }_{11}
    \underbrace{C(x_{i_r},g(u_r*v))^{-1}}_{12}
    \underbrace{C(x_{\tau},g(u_r*v))}_{13}
    \underbrace{C(x_{\tau},v)^{-1}}_{14}
    \underbrace{C(u,x_{\tau})}_{15}
    \underbrace{C(x_{i_r},v)}_{16}
    \underbrace{C(u,v)}_{17}
    \underbrace{C(x_{\tau},x_{i_r})^{-1}}_{18}\\
    &\cdot \underbrace{C(x_{\tau},u)^{-1}}_{19}
    \underbrace{C(x_{\sigma},x_{i_r})^{-1}}_{20}
    \underbrace{C(x_{\sigma},u)^{-1}}_{21}
    \underbrace{C(x_{i_r},u)}_{22}
    \underbrace{C(g(u_r*v),x_{i_r})^{-1}}_{23}
    C(g(u_r*v),u)^{-1}
    C(g(u_r*v),u_r)
\end{align*}
\normalsize
correspond to the following labeled factors in \cref{eq:Prod4}:
\small
\begin{align*}
    & \underbrace{C(x_{\sigma},x_{\tau})}_{10}
    \underbrace{C(x_{\sigma},x_{i_r})^{-1}}_{20}
    \underbrace{C(u,x_{\tau})}_{15}
    \underbrace{C(x_{\sigma},g(u*v_s))}_{11}
    \underbrace{C(x_{\sigma},u)^{-1}}_{21}
    \underbrace{C(x_{\tau},g(u*v_s))}_{13}
    \underbrace{C(x_{j_s},g(u*v_s))^{-1}}_{12}
    \underbrace{C(x_{\tau},u)^{-1}}_{19}\\
    &\cdot \underbrace{C(x_{j_s},u)}_{22}
    \underbrace{C(x_{\tau},x_{j_s})^{-1}}_{18}
    \underbrace{C(x_{\tau},v)^{-1}}_{14}
    \underbrace{C(x_{j_s},v)}_{16}
    \underbrace{C(u,v)}_{17}
    \underbrace{C(g(u*v_s),x_{j_s})^{-1}}_{23}
    C(g(u*v_s),v)^{-1}
    C(g(u*v_s),v_s).
\end{align*}
\normalsize
Thus it remains to show that 
\[
C(g(u_r*v),u)^{-1} C(g(u_r*v),u_r) = C(g(u*v_s),v)^{-1} C(g(u*v_s),v_s),
\]
which follows since 
\begin{align*}
    1 &= C\left(g(u_{r}*v), \frac{u_{r}*v}{g(u_{r}*v)}\right)
    C\left(g(u*v_{s}), \frac{u*v_{s}}{g(u*v_{s})}\right)^{-1} \textnormal{ by \cref{rem:C=1}}\\
    &= C\left(g(g(x_{i_{r}}*u)*v), \frac{u_{r}*v}{g(u_{r}*v)}\right)
    C\left(g(u*g(x_{j_s}*v)), \frac{u*v_{s}}{g(u*v_{s})}\right)^{-1}\\
    &= C\left(g(x_{i_{r}}*u*v), \frac{u_{r}*v}{g(u_{r}*v)}\right)
    C\left(g(x_{j_s}*u*v), \frac{u*v_{s}}{g(u*v_{s})}\right)^{-1} \textnormal{ by \cref{rem:DecFun}(2)}\\
    &= C\left(g(x_{i_{r}}*u*v), \frac{\frac{u_{r}*v}{g(u_{r}*v)}}{\frac{u*v_{s}}{g(u*v_{s})}}\right)\\
    &= C\left(g(x_{i_{r}}*u*v), \frac{u_{r}*v}{u*v_{s}}\right)\\
    &= C(g(x_{i_{r}}*u*v), u_{r})
    C(g(x_{i_{r}}*u*v), v)
    C(g(x_{i_{r}}*u*v), u)^{-1}
    C(g(x_{i_{r}}*u*v), v_{s})^{-1}\\
    &= C(g(u_{r}*v), u_{r})
    C(g(u*v_{s}), v)
    C(g(u_{r}*v), u)^{-1}
    C(g(u*v_{s}), v_{s})^{-1}.
\end{align*}
The proof that the scalar coefficient of $e(\sigma_{r}*\tau;g(u*v))$ in \cref{eq:Prod1} coincides with the scalar coefficient of $e(\sigma*\tau_{s};g(u*v))$ in \cref{eq:Prod2} is similar and is therefore omitted.

\textbf{Case 2}: We assume that $\sigma$ and $\tau$ do not have an index in common.

Recall that we want to show
\begin{equation}\label{eq:Leibniz}
d(e(\sigma;u)e(\tau;v)) = d(e(\sigma;u))e(\tau;v) + (-1)^{|e(\sigma;u)|} e(\sigma;u)d(e(\tau;v)).
\end{equation}
When we expand $d(e(\sigma;u)e(\tau;v))$, we get 

\begin{align}\label{eq:DiffProd}
    \!\!\!\!\!\!\!\!\!\!\!\!\!\!\!\!\!\!\!\!d\left(e(\sigma*\tau; g(u*v)) 
    \chi(u,x_{\tau}) 
    C(x_{\sigma}, x_{\tau})
    C(u,v)
    C\left( x_{\sigma}, \frac{g(u*v)}{u} \right)
    C\left( x_{\tau}, \frac{g(u*v)}{v} \right)
    \frac{u*v}{g(u*v)}
    \right)
\end{align}
\begin{align}
    = &C(x_{\sigma}*u,x_{\tau}*v)
    C\left(x_{\sigma}*x_{\tau},\frac{u*v}{g(u*v)}\right)^{-1}
    d(e(\sigma*\tau; g(u*v)))
    \frac{u*v}{g(u*v)}\nonumber\\
    = &C(x_{\sigma}*u,x_{\tau}*v)
    C\left(x_{\sigma}*x_{\tau},\frac{u*v}{g(u*v)}\right)^{-1}
    \sum_{r=1}^{l_{\sigma}+l_{\tau}}
    (-1)^{r}
    e((\sigma*\tau)_{r};g(u*v))\nonumber\\
    &\cdot C\left(x_{(\sigma*\tau)_{r}}*g(u*v), \frac{x_{\sigma*\tau}}{x_{(\sigma*\tau)_{r}}}\right)^{-1}
    C\left(\frac{x_{\sigma*\tau}}{x_{(\sigma*\tau)_{r}}}, \frac{u*v}{g(u*v)} \right)
    \frac{x_{\sigma*\tau}}{x_{(\sigma*\tau)_{r}}}*
    \frac{u*v}{g(u*v)}\nonumber\\
    - &C(x_{\sigma}*u,x_{\tau}*v)     C\left(x_{\sigma}*x_{\tau},\frac{u*v}{g(u*v)}\right)^{-1}
    \sum_{r=1}^{l_{\sigma}+l_{\tau}}
    (-1)^{r}
    e\left((\sigma*\tau)_{r};g\left(\frac{x_{\sigma*\tau}}{x_{(\sigma*\tau)_{r}}}*g(u*v)\right)\right)\nonumber\\
    &\cdot C\left(x_{(\sigma*\tau)_{r}}, \frac{\frac{x_{\sigma*\tau}}{x_{(\sigma*\tau)_{r}}}*g(u*v)}{g\left(\frac{x_{\sigma*\tau}}{x_{(\sigma*\tau)_{r}}}*g(u*v)\right)}\right)^{-1}
    C\left(\frac{\frac{x_{\sigma*\tau}}{x_{(\sigma*\tau)_{r}}}*g(u*v)}{g\left(\frac{x_{\sigma*\tau}}{x_{(\sigma*\tau)_{r}}}*g(u*v)\right)}, \frac{u*v}{g(u*v)} \right)\\
    &\frac{\frac{x_{\sigma*\tau}}{x_{(\sigma*\tau)_{r}}}*g(u*v)}{g\left(\frac{x_{\sigma*\tau}}{x_{(\sigma*\tau)_{r}}}*g(u*v)\right)}
    *\frac{u*v}{g(u*v)}\nonumber
\end{align}
where $l_{\sigma}$, $l_{\tau}$ are the lengths of $\sigma$ and $\tau$, respectively. Now we expand the righthand side of \cref{eq:Leibniz}. Expanding $d(e(\sigma;u))$ gives
\begin{align*}
    &\sum_{r=1}^{l_{\sigma}}
    e(\sigma_{r};u)
    (-1)^{r}
    C\left(x_{\sigma_{r}}*u, \frac{x_{\sigma}}{x_{\sigma_{r}}}\right)^{-1}
    \frac{x_{\sigma}}{x_{\sigma_{r}}}
    - 
    \sum_{r=1}^{l_{\sigma}}
    (-1)^{r}
    e\left(\sigma_{r};g\left(\frac{x_{\sigma}}{x_{\sigma_{r}}}*u\right)\right)
    C\left(x_{\sigma_{r}},
    \frac{\frac{x_{\sigma}}{x_{\sigma_{r}}}*u}{g\left(\frac{x_{\sigma}}{x_{\sigma_{r}}}*u\right)}\right)^{-1}
    \frac{\frac{x_{\sigma}}{x_{\sigma_{r}}}*u}{g\left(\frac{x_{\sigma}}{x_{\sigma_{r}}}*u\right)}
\end{align*}
and after multiplying by the appropriate scalars to move $e(\tau;v)$ past the monomial coefficients, multiplying the equation above by $e(\tau;v)$ on the right gives
\begin{align*}
    &\sum_{r=1}^{l_{\sigma}}
    e(\sigma_{r};u)
    e(\tau;v)
    (-1)^{r}
    C\left(x_{\sigma_{r}}*u, \frac{x_{\sigma}}{x_{\sigma_{r}}}\right)^{-1}
    \chi\left(\frac{x_{\sigma}}{x_{\sigma_{r}}}, x_{\tau}*v\right)
    \frac{x_{\sigma}}{x_{\sigma_{r}}}\\
    - &\sum_{r=1}^{l_{\sigma}}
    (-1)^{r}
    e\left(\sigma_{r};g\left(\frac{x_{\sigma}}{x_{\sigma_{r}}}*u\right)\right)
    e(\tau;v)
    C\left(x_{\sigma_{r}},
    \frac{\frac{x_{\sigma}}{x_{\sigma_{r}}}*u}{g(\frac{x_{\sigma}}{x_{\sigma_{r}}}*u)}\right)^{-1}
    \chi\left(\frac{\frac{x_{\sigma}}{x_{\sigma_{r}}}*u}{g(\frac{x_{\sigma}}{x_{\sigma_{r}}}*u)}, x_{\tau}*v\right)
    \frac{\frac{x_{\sigma}}{x_{\sigma_{r}}}*u}{g(\frac{x_{\sigma}}{x_{\sigma_{r}}}*u)}.
\end{align*}
Using the formula for the product, the previous display becomes
\begin{align}\label{eq:ProdDiff1}
    &\sum_{r=1}^{l_{\sigma}}
    (-1)^{r}
    e(\sigma_{r}*\tau;g(u*v))
    \chi(u,x_{\tau})
    C(x_{\sigma_{r}},x_{\tau})
    C(u,v)
    C\left( x_{\sigma_{r}}, \frac{g(u*v)}{u} \right)
    C\left( x_{\tau}, \frac{g(u*v)}{v} \right)\\
    &\cdot C\left(x_{\sigma_{r}}*u, \frac{x_{\sigma}}{x_{\sigma_{r}}}\right)^{-1}
    \chi\left(\frac{x_{\sigma}}{x_{\sigma_{r}}}, x_{\tau}*v\right)
    C\left(\frac{u*v}{g(u*v)}, \frac{x_{\sigma}}{x_{\sigma_{r}}}\right)
    \frac{u*v}{g(u*v)}*
    \frac{x_{\sigma}}{x_{\sigma_{r}}}\nonumber
\end{align}
\begin{align*}
    \hspace{0.45in} -&\sum_{r=1}^{l_{\sigma}}
    (-1)^{r}
    e\left(\sigma_{r}*\tau;g\left(g\left(\frac{x_{\sigma}}{x_{\sigma_{r}}}*u\right)*v\right)\right)
    \chi\left(g\left(\frac{x_{\sigma}}{x_{\sigma_{r}}}*u\right), x_{\tau}\right)
    C(x_{\sigma_{r}}, x_{\tau})
    C\left(g\left(\frac{x_{\sigma}}{x_{\sigma_{r}}}*u\right), v\right)\nonumber\\
    &\cdot C\left( x_{\sigma_{r}}, \frac{g(g(\frac{x_{\sigma}}{x_{\sigma_{r}}}*u)*v)}{g(\frac{x_{\sigma}}{x_{\sigma_{r}}}*u)}
    \right)
    C\left( x_{\tau}, \frac{g(g(\frac{x_{\sigma}}{x_{\sigma_{r}}}*u)*v)}{v} \right)
    C\left(x_{\sigma_{r}},
    \frac{\frac{x_{\sigma}}{x_{\sigma_{r}}}*u}{g(\frac{x_{\sigma}}{x_{\sigma_{r}}}*u)}\right)^{-1}\\
    &\cdot \chi\left(\frac{\frac{x_{\sigma}}{x_{\sigma_{r}}}*u}{g(\frac{x_{\sigma}}{x_{\sigma_{r}}}*u)}, x_{\tau}*v\right)\nonumber
    C\left( \frac{g(\frac{x_{\sigma}}{x_{\sigma_{r}}}*u)*v}{g(g(\frac{x_{\sigma}}{x_{\sigma_{r}}}*u)*v)}, \frac{\frac{x_{\sigma}}{x_{\sigma_{r}}}*u}{g(\frac{x_{\sigma}}{x_{\sigma_{r}}}*u)}
    \right)
    \frac{g(\frac{x_{\sigma}}{x_{\sigma_{r}}}*u)*v}{g(g(\frac{x_{\sigma}}{x_{\sigma_{r}}}*u)*v)}*
    \frac{\frac{x_{\sigma}}{x_{\sigma_{r}}}*u}{g(\frac{x_{\sigma}}{x_{\sigma_{r}}}*u)}\nonumber.
\end{align*}
There are $l_{\sigma}$ choices of $r$ for which $(\sigma*\tau)_{r} = \sigma_{r}*\tau$. At these $r$, the basis elements $e((\sigma*\tau)_{r};g(u*v))$ in \eqref{eq:DiffProd} and $e(\sigma_{r}*\tau;g(u*v))$ in \cref{eq:ProdDiff1} coincide. We check that the monomials coefficients coincide. The basis element $e((\sigma*\tau)_{r};g(u*v))$ from \Cref{eq:DiffProd} is multiplied by the monomial
\begin{align*}
    \frac{x_{\sigma*\tau}}{x_{(\sigma*\tau)_{r}}} *
    \frac{u*v}{g(u*v)},
\end{align*}
which, by our choice of $r$, is equal to
\begin{align*}
\frac{u*v}{g(u*v)}*
    \frac{x_{\sigma}}{x_{\sigma_{r}}}
\end{align*}
which is the monomial coefficient of $e(\sigma_{r}*\tau;g(u*v))$ in \cref{eq:ProdDiff1}. Now we check that the scalar coefficients coincide as well. After using the property that $C$ and $\chi$ are bicharacters and rewriting $x_{\sigma*\tau} = x_{\sigma}*x_{\tau}$ and $x_{(\sigma*\tau)_{r}} = x_{\sigma_{r}}*x_{\tau}$, the scalar coefficient of $e((\sigma*\tau)_{r};g(u*v))$ from \Cref{eq:DiffProd} is
\begin{align*}
    &C(x_\sigma,x_\tau)
    \underbrace{C(x_\sigma,v)}_{1}
    C(u,x_\tau*v)
    \underbrace{C(x_\sigma,u)^{-1}}_{2}
    \underbrace{C(x_\sigma,v)^{-1}}_{1}
    \underbrace{C(x_\sigma,g(u*v))}_{3}
    \underbrace{C\left(x_\tau,\frac{u*v}{g(u*v)}\right)^{-1}}_{4}
    C(x_{\sigma_{r}},x_{\sigma})^{-1}\\
    &\cdot\underbrace{C(x_{\sigma_{r}},x_{\tau})^{-1}}_{5}
    C(x_{\tau},x_{\sigma})^{-1}
    \underbrace{C(x_{\tau},x_{\tau})^{-1}}_{6}
    C(x_{\sigma_{r}},x_{\sigma_{r}})
    \underbrace{C(x_{\sigma_{r}},x_{\tau})}_{5}
    C(x_{\tau},x_{\sigma_{r}})
    \underbrace{C(x_{\tau},x_{\tau})}_{6}
    C(g(u*v),x_{\sigma})^{-1}\\
    &\cdot\underbrace{C(g(u*v),x_{\tau})^{-1}}_{7}
    C(g(u*v), x_{\sigma_{r}})
    \underbrace{C(g(u*v), x_{\tau})}_{7}
    \underbrace{C(x_{\sigma},u)}_{2}
    C(x_{\sigma},v)
    \underbrace{C(x_{\sigma}, g(u*v))^{-1}}_{3}
    \underbrace{C\left(x_{\tau},\frac{u*v}{g(u*v)}\right)}_{4}\\
    &\cdot C(x_{\sigma_{r}},u)^{-1}
    C(x_{\sigma_{r}},v)^{-1}
    C(x_{\sigma_{r}},g(u*v))
    C\left(x_{\tau},\frac{u*v}{g(u*v)}\right)^{-1}.
\end{align*}
The factors with the same label are inverses of each other; once simplified the previous display reduces to
\begin{align}\label{eq:Prod7}
    &C(x_\sigma,x_\tau)
    C(u,x_\tau*v)
    C(x_{\sigma_{r}},x_{\sigma})^{-1}
    C(x_{\tau},x_{\sigma})^{-1}
    C(x_{\sigma_{r}},x_{\sigma_{r}})
    C(x_{\tau},x_{\sigma_{r}})
    C(g(u*v),x_{\sigma})^{-1}
    C(g(u*v), x_{\sigma_{r}})\\
    &\cdot C(x_{\sigma},v)
    C(x_{\sigma_{r}},u)^{-1}
    C(x_{\sigma_{r}},v)^{-1}
    C(x_{\sigma_{r}},g(u*v))
    C\left(x_{\tau},\frac{u*v}{g(u*v)}\right)^{-1}.\nonumber
\end{align}
Similarly, the scalar coefficient of $e(\sigma_{r}*\tau;g(u*v))$ from \cref{eq:ProdDiff1} simplifies to
\begin{align}\label{eq:Prod8}
    &C(u,x_\tau*v)
    C(x_{\sigma_r},u)^{-1}
    C(x_{\sigma_r},g(u*v))
    C\left(x_\tau,\frac{u*v}{g(u*v)}\right)^{-1}
    C(x_{\sigma_{r}}, x_{\sigma})^{-1}
    C(x_{\sigma_{r}}, x_{\sigma_{r}})\\
    &\cdot C(x_{\sigma},x_{\tau})
    C(x_{\tau},x_{\sigma})^{-1}
    C(x_{\sigma},v)
    C(x_{\tau},x_{\sigma_r})
    C(x_{\sigma_r},v)^{-1}
    C(g(u*v), x_{\sigma})^{-1}
    C(g(u*v), x_{\sigma_{r}}).\nonumber
\end{align}
We notice that the following labeled factors in \cref{eq:Prod7}
\begin{align*}
    &\underbrace{C(x_\sigma,x_\tau)}_{8}
    \underbrace{C(u,x_\tau*v)}_{9}
    \underbrace{C(x_{\sigma_{r}},x_{\sigma})^{-1}}_{10}
    \underbrace{C(x_{\tau},x_{\sigma})^{-1}}_{11}
    \underbrace{C(x_{\sigma_{r}},x_{\sigma_{r}})}_{12}
    \underbrace{C(x_{\tau},x_{\sigma_{r}})}_{13}
    \underbrace{C(g(u*v),x_{\sigma})^{-1}}_{14}
    \underbrace{C(g(u*v), x_{\sigma_{r}})}_{15}\\
    &\cdot \underbrace{C(x_{\sigma},v)}_{16}
    \underbrace{C(x_{\sigma_{r}},u)^{-1}}_{17}
    \underbrace{C(x_{\sigma_{r}},v)^{-1}}_{18}
    \underbrace{C(x_{\sigma_{r}},g(u*v))}_{19}
    \underbrace{C\left(x_{\tau},\frac{u*v}{g(u*v)}\right)^{-1}}_{20}
\end{align*}
correspond to the following labeled factors in \cref{eq:Prod8}
\begin{align*}
    &\underbrace{C(u,x_\tau*v)}_{9}
    \underbrace{C(x_{\sigma_r},u)^{-1}}_{17}
    \underbrace{C(x_{\sigma_r},g(u*v))}_{19}
    \underbrace{C\left(x_\tau,\frac{u*v}{g(u*v)}\right)^{-1}}_{20}
    \underbrace{C(x_{\sigma_{r}}, x_{\sigma})^{-1}}_{10}
    \underbrace{C(x_{\sigma_{r}}, x_{\sigma_{r}})}_{12}\\
    &\cdot \underbrace{C(x_{\sigma},x_{\tau})}_{8}
    \underbrace{C(x_{\tau},x_{\sigma})^{-1}}_{11}
    \underbrace{C(x_{\sigma},v)}_{16}
    \underbrace{C(x_{\tau},x_{\sigma_r})}_{13}
    \underbrace{C(x_{\sigma_r},v)^{-1}}_{18}
    \underbrace{C(g(u*v), x_{\sigma})^{-1}}_{14}
    \underbrace{C(g(u*v), x_{\sigma_{r}})}_{15}.
\end{align*}
Hence the coefficient of $e((\sigma*\tau)_{r};g(u*v))$ in \cref{eq:DiffProd} is the same as the coefficient of $e(\sigma_{r}*\tau;g(u*v))$ in \cref{eq:ProdDiff1} for the $r$ for which $(\sigma*\tau)_r=\sigma_r*\tau$. By a similar argument, for these same $r$, the coefficient of $e\left((\sigma*\tau)_{r};g\left( \frac{x_{\sigma*\tau}}{x_{(\sigma*\tau)_{r}}}*g(u*v)\right) \right)$ in \cref{eq:DiffProd} and the coefficient of  $e\left(\sigma_{r}*\tau;g\left(g\left(\frac{x_{\sigma}}{x_{\sigma_{r}}}*u\right)*v\right)\right)$ from \cref{eq:ProdDiff1} coincide.

It remains to show that for the $r$ for which $(\sigma*\tau)_r=\sigma*\tau_r$ the coefficients of $e((\sigma*\tau)_r,g(u*v))$ and $e\left((\sigma*\tau)_{r};g\left(\frac{x_{\sigma*\tau}}{x_{(\sigma*\tau)_{r}}}*g(u*v)\right)\right)$ in \Cref{eq:DiffProd} are the same as the coefficients of $e(\sigma*\tau_r,g(u*v))$ and $e\left(\sigma*\tau_r,g\left(u*g\left(\frac{x_\tau}{x_{\tau_r}}*v\right)\right)\right)$, respectively, in the expansion of $(-1)^{|e(\sigma;u)|}e(\sigma,u)d(e(\tau,v))$. This is a long and tedious computation analogous to the one just performed, and it is therefore omitted. \qedhere
\end{proof}

\begin{example}
The multiplication table of the skew Eliahou-Kervaire resolution of the ideal considered in \cref{ex:Example} is

\begin{center}
\def\arraystretch{1.25}
\begin{tabular}{|c||c|c|}
\hline
& $e(1;xy)$ & $e(1;y^2)$\\
\hline\hline
$e(\emptyset;x^2)$ & $e(1;x^2)xy$ & $e(1;x^2)y^2$\\
\hline
$e(\emptyset;xy)$ & $e(1,x^2)q^{-2}y^2$ & $e(1;xy)q^{-1}y^2$\\
\hline
$e(\emptyset;y^2)$ & $e(1;xy)q^{-4}y^2$ & $e(1;y^2)q^{-2}y^2$\\
\hline
\end{tabular}
\end{center}
where the basis elements in the leftmost column are taken to be the first factor in the multiplication. All other products are 0.
\end{example}

\bibliographystyle{amsplain}
\bibliography{biblio}

\end{document}